\documentclass[12pt,a4paper]{amsart}

\usepackage{amsmath}
\usepackage{amssymb}
\usepackage{amsthm}
\usepackage{color}
\usepackage{enumerate}
\usepackage{float}
\usepackage{wrapfig}
\usepackage{layout}
\usepackage{comment}

\renewcommand{\baselinestretch}{1.4}

\newcommand{\RN}{\mathbb{R}^N}

\DeclareMathOperator\USC{USC}
\DeclareMathOperator\LSC{LSC}
\DeclareMathOperator\dist{dist}

\DeclareMathOperator\Lip{Lip}

\def\cO{\mathcal{O}}

\def\R{\mathbb{R}}
\def\Q{\mathbb{Q}}
\def\Z{\mathbb{Z}}
\def\N{\mathbb{N}}
\def\cF{\mathcal{F}}
\def\cP{\mathcal{P}}
\def\cL{\mathcal{L}}

\def\cS{\mathcal{S}}
\def\cA{\mathcal{A}}
\def\bye{\end{document}}
\def\by{\end{proof}\bye}

\def\hello{\begin{document}}
\def\fr{\frac} 
\def\disp{\displaystyle}  
\def\ga{\alpha}     
\def\go{\omega}
\def\gep{\varepsilon}      
\def\ep{\gep}    
\def\mid{\,:\,}   
\def\gb{\beta} 
\def\gam{\gamma}
\def\gd{\delta}
\def\gz{\zeta} 
\def\gth{\theta}   
\def\gk{\kappa} 
\def\gl{\lambda}
\def\gL{\Lambda}
\def\gs{\sigma}   
\def\gf{\varphi}                  
\def\tim{\times}                        
\def\aln{&\,}
\def\ol{\overline}
\def\ul{\underline}           
\def\pl{\partial}
\def\hb{\text}                
\def\cF{\mathcal{F}}
\def\Int{\mathop{\text{int}}}

\def\gG{\varGamma}
\def\lan{\langle}
\def\ran{\rangle}
\def\cD{\mathcal{D}}
\def\cB{\mathcal{B}}
\def\bcases{\begin{cases}}
\def\ecases{\end{cases}}
\def\balns{\begin{align*}}
\def\ealns{\end{align*}}
\def\balnd{\begin{aligned}}
\def\ealnd{\end{aligned}}
\def\1{\mathbf{1}}
\def\bproof{\begin{proof}}

\def\eproof{\end{proof}}

\theoremstyle{definition}
\newtheorem{definition}{Definition}%[section]
\theoremstyle{plain}
\newtheorem{theorem}[definition]{Theorem}
\newtheorem{corollary}[definition]{Corollary}
\newtheorem{lemma}[definition]{Lemma}
\newtheorem{proposition}[definition]{Proposition}
\theoremstyle{remark}
\newtheorem{remark}[definition]{Remark}
\newtheorem{notation}[definition]{Notation}

\def\red#1{\textcolor{red}{#1}}
\def\blu#1{\textcolor{blue}{#1}}
\def\beq{\begin{equation}}
\def\eeq{\end{equation}}
\def\bthm{\begin{theorem}}
\def\ethm{\end{theorem}}
\def\bproof{\begin{proof}}
\def\eproof{\end{proof}}
\renewcommand{\baselinestretch}{1.4} 
\def\R{\mathbb{R}}
\def\Q{\mathbb{Q}}
\def\Z{\mathbb{Z}}
\def\N{\mathbb{N}}
\def\cF{\mathcal{F}}
\def\cP{\mathcal{P}}
\def\cL{\mathcal{L}} 
\def\cG{\mathcal{G}}
\def\cS{\mathcal{S}}
\def\cA{\mathcal{A}}

\def\disp{\displaystyle}       
\def\mid{\,:\,}   
\def\gl{\lambda}
\def\gL{\Lambda}
\def\gs{\sigma}   
\def\gf{\varphi}                  
\def\tim{\times}                        
\def\ol{\overline}
\def\ul{\underline}           
\def\pl{\partial}
\def\hb{\text}                
\def\cF{\mathcal{F}}
\def\Int{\mathop{\text{int}}}

\def\lan{\langle}
\def\ran{\rangle}
\def\cD{\mathcal{D}}
\def\cB{\mathcal{B}}
\def\bcases{\begin{cases}}
\def\ecases{\end{cases}}
\def\balns{\begin{align*}}
\def\ealns{\end{align*}}
\def\balnd{\begin{aligned}}
\def\ealnd{\end{aligned}}
\def\1{\mathbf{1}}

\def\red#1{\textcolor{red}{#1}}
\def\blu#1{\textcolor{blue}{#1}}
\def\beq{\begin{equation}}  
\def\eeq{\end{equation}}
\def\bthm{\begin{theorem}}
\def\ethm{\end{theorem}}
\def\bproof{\begin{proof}}
\def\eproof{\end{proof}}

\def\otim{\otimes}
\def\bald{\begin{aligned}}
\def\eald{\end{aligned}}
\def\stm{\setminus}
\long\def\/*#1*/{}
\def\Gth{\varTheta}
\def\gD{\varDelta}
\def\gG{\varGamma}
\def\cW{\mathcal{W}}
\def\what{\widehat}
\def\hM{\what M}
\def\eqr#1{\eqref{#1}}
\def\tin{\hbox{ in }}
\def\ton{\hbox{ on }}
\def\erf{\eqref}
\def\AND{\hbox{ and }}
\def\thinsk{{\vrule height0pt width1pt depth 0pt}}
\def\bmat{\begin{pmatrix}}
\def\emat{\end{pmatrix}}

\numberwithin{equation}{section}

\def\ol{\overline}
\newcommand{\vep}{\varepsilonroof}
\newcommand{\vro}{\varrho}
\def\vtheta{\vartheta}
\newcommand{\un}{u_n}
\newcommand{\uzn}{u_{0n}}
\newcommand{\fn}{f_n}
\newcommand{\um}{u_m}
\def\H{{\rm ({\bf H})}}
\def\grad{\nabla}
\newcommand{\refe}[1]{{(\ref{#1})}}
\newcommand{\dis}{\displaystyle}
\newcommand{\lbl}[1]{\label{#1}}
\newcommand{\convd}{\rightharpoonup}
\newcommand{\noi}{\noindent}
\newcommand{\dive}{{\rm div}}
\newcommand{\mis}{{\rm meas}}
\newcommand{\Mm}{\mathcal{M}^-_{\lambda, \Lambda}}
\newcommand{\Mp}{\mathcal{M}^+_{\lambda, \Lambda}}
\newcommand{\Pmo}{\mathcal{P}^-_{1}}
\newcommand{\Ppo}{\mathcal{P}^+_{1}}
\newcommand{\Pmk}{\mathcal{P}^-_{k}}
\newcommand{\Ppk}{\mathcal{P}^+_{k}}
\newcommand{\cC}{\mathcal{C}}
\newcommand{\uog}{u^\omega_\gamma}
\newcommand{\Oog}{\Omega^\omega_\gamma}
\newcommand{\fiog}{\phi^\omega_\gamma}
\newcommand{\Ooga}{\Omega^\omega_{\gamma, a}}
\newcommand{\uoga}{u^\omega_{\gamma, a}}
\DeclareMathOperator{\Conv}{Conv}
\DeclareMathOperator{\Span}{Span}
\DeclareMathOperator{\INT}{int}
\begin{document}
\parindent=0pt

\title[\textbf{Truncated Laplacians}]{\textbf{Existence through convexity for the\\
truncated Laplacians}}

\author[ ]{I. Birindelli, G. Galise, H. Ishii}

\address[\textsc{I. Birindelli}]{Dipartimento di Matematica\newline
\indent Sapienza Universit\`a  di Roma \newline
 \indent   P.le Aldo  Moro 2, I--00185 Roma, Italy.\newline
\indent E-mail: isabeau@mat.uniroma1.it}

\address[\textsc{G. Galise}]{Dipartimento di Matematica\newline
\indent Sapienza Universit\`a  di Roma \newline
 \indent   P.le Aldo  Moro 2, I--00185 Roma, Italy.\newline
\indent E-mail: galise@mat.uniroma1.it}

\address[\textsc{H. Ishii}]{Institute for Mathematics and Computer Science\newline
\indent Tsuda University  \newline
 \indent   2-1-1 Tsuda-machi, Kodaira-shi, Tokyo, 187-8577 Japan.\newline
\indent E-mail: hitoshi.ishii@waseda.jp}
\thanks{Part of this work was done while the third author 
was visiting the Dipartimento di Matematica,
Sapienza Universit\`a  di Roma in May, 2018. 
He would like to thank the department for its hospitality and financial support. 
His work was also partially supported by the KAKENHI \#26220702, \#16H03948,
\#18H00833, JSPS}
% \email{isabeau@mat.uniroma1.it, galise@mat.uniroma1.it}
%\maketitle
\begin{abstract}
We study the Dirichlet problem on a bounded convex domain of $\R^N$, with zero boundary data, for truncated Laplacians $\cP_k^\pm$, with $k<N$. 
We establish a necessary and sufficient condition (Theorem 1) in terms of the ``flatness'' of domains for existence of a solution for general inhomogeneous term. 
This result, in particular, shows that the strict convexity of the domain is sufficient 
for the solvability of the Dirichlet problem. 
The result and related 
ideas are applied to the solvability of the Dirichlet problem for the operator 
$\cP_k^+$ with lower order term when the domain is strictly convex and the existence of principal eigenfunctions for the operator $\cP_1^+$.  An existence theorem  is presented with regard to the principal eigenvalue for  
the Dirichlet problem with zero-th order term for the operator $\cP_1^+$. A nonexistence result is established for the operator $\cP_k^+$ with 
first order term when the domain has a boundary portion which is nearly flat.  
Furthermore, when the domain is a ball, we study the Dirichlet problem, 
with a constant inhomogeneous term and a possibly sign-changing first order term,
and the associated eigenvalue problem. 
%The results are on the solvability of these problems, with a detailed analysis on the corresponding ODE. 
\end{abstract}

\maketitle

\section{Introduction}
For any $N\times N$ symmetric matrix $X$, let 
\begin{equation}\label{arranged eigenvalues}
\lambda_1(X)\leq \lambda_2(X)\leq\cdots\leq \lambda_N(X) 
\end{equation}
be the ordered eigenvalues of $X$. For $k\in[1,N]$, $k$ integer, let 
\begin{equation}\label{def1}
\Pmk(D^2u)=\sum_{i=1}^k\lambda_i(D^2u) \qquad\text{and}\qquad \Ppk(D^2u)=\sum_{i=1}^k\lambda_{N+1-i}(D^2u).
\end{equation}
For $k=N$ these operators coincide with the Laplacian, hence we will always consider $k<N$. 

In the whole paper $\Omega$ will be a bounded domain of $\R^N$. The scope  
of the paper is to study existence of solutions for the following Dirichlet problem 
\begin{equation}\label{Pb00}
\left\{\begin{array}{cl}
\Ppk(D^2u)+ H(x,Du)= f(x) & \text{in $\Omega$}\\
u=0 & \text{on $\partial\Omega$}.
\end{array}\right.
\end{equation}
Throughout this paper, the Dirichlet boundary condition is understood in the classical pointwise sense.
Before %continuing 
describing the result of this paper, let us mention that the operators 
$\Ppk$ and $\Pmk$ come out naturally in geometrical problems in particular when considering 
manifolds of partially positive curvature, see \cite{Sha,Wu}, or mean curvature flow in arbitrary codimension, see \cite{AS}.
Lately the interest has been from a pure PDE theoretical point of view, starting from the works of Harvey and Lawson
\cite{HL1,HL2} and Caffarelli, Li and Nirenberg \cite{CLN1} continuing with \cite{OS} by Oberman and L. Silvestre on 
convex envelope.  See also \cite{CDLV, AGV, BGL, G, GV} for further contributions.\\Some analogies can be found in the work of Blanc and Rossi \cite{BR} but we will be more explicit about their work at the end of the introduction.

In \cite{BGI}, when $\Omega$ is uniformly convex, i.e. when there exists $R>0$ and $Y\subseteq\mathbb R^N$ such that 
\begin{equation}\label{u-convex}\Omega=\bigcap_{y\in Y}B_R(y)\eeq
we called these domains hula hoop domains and, in these domains we proved existence of solutions
for any bounded $f$ as long as $|H(x,p)-H(x,q)|\leq b|p-q|$ and $bR<k$.  

On the other hand, in \cite{BGI1}, if $\Omega$ is only convex, i.e. an intersection of half spaces or cubes, $k=1$ and 
$H\equiv 0$, existence was established under some sign 
condition on $f$ near the boundary of $\Omega$.

In a general sense we wish to understand up to which point these conditions are optimal.
We will see how these degenerate elliptic operators are extremely sensitive to the \lq\lq convexity\rq\rq 
of the domain and are strongly influenced by the presence of the first order term. 

In fact, in order to concentrate on the domain, we shall treat first the case where $H(x,D u)\equiv 0$. 

In a first step we shall see that convexity alone, does not allow to prove existence of supersolutions for any $f$. 
In order to solve the Dirichlet problem with general right hand side $f$ we should  impose that $\partial\Omega$ has at least $N-k$ directions of strict convexity.  We are now going to be more precise.

We can introduce a sort of \lq\lq classification\rq\rq\  of strict convexity.

Consider for $j=1,\ldots,N$ 
\begin{equation}\label{j strict convex}
C_j=\left\{C\subset\RN\;:\;C=\omega\times\R^{N-j},\;\omega\subset\R^j\;\,\text{bounded and strictly convex}\right\}.
\end{equation}
Henceforth we denote by $\mathcal{C}_j$ the class of all convex and bounded domains $\Omega\subset\RN$ which are intersection (up to rotations) of cylinders belonging to $C_j$. More precisely  $\Omega\in\mathcal{C}_j$ if, and only if, 
for each $x\in\pl\Omega$, there exist 
$O\in\cO^N$, with $\cO^N$ being the class of orthogonal $N\tim N$ matrices,
and $C\in C_j$ such that 
\begin{equation}\label{j strict convex 1}
\Omega\subset OC \ \ \text{ and }  \ \ x\in\pl(OC).
\end{equation}
We denote by $S_j=S_j(\Omega)$ the set of all $(O,C)\in\mathcal{O}^N\tim C_j$ such that 
for some $x\in \pl\Omega$, \eqref{j strict convex 1} is satisfied. 
%there exists subset $S\subset \mathcal{O}^N\times C_j$, $\mathcal{O}^N$ being the class of orthogonal $N \times N$ matrices, such that 
One has 
\begin{equation}\label{j strict convex 2}
\Omega=\bigcap_{(O,C)\in S_j}OC \ \ \text{ if }\ \Omega\in\cC_j,
\eeq
and 
$$
\mathcal{C}_1\supset\mathcal{C}_2\supset\ldots\supset\mathcal{C}_N.
$$
%\begin{equation}\label{j strict convex 2}
%\Omega=\bigcap_{(O,C)\in S}OC.
%\end{equation}
%One has 
%$$
%\mathcal{C}_1\supset\mathcal{C}_2\supset\ldots\supset\mathcal{C}_N.
%$$
Note that $\mathcal{C}_1$ and $\mathcal{C}_N$ correspond respectively to the class of bounded convex   and strictly convex domains.  
It may be useful to note that if $\omega\subset\R^j$, $C\subset \R^N$, 
and $O\in\cO^N$, then 
\[
\pl (\omega\tim\R^{N-j})=\pl\omega\tim\R^{N-j} \ \ \AND 
\ \ \pl (OC)=O\pl C. 
\]
It might be remarked at this point  that, when $\Omega$ is given by \erf{u-convex}, one can find 
$y\in \ol Y$ for each $x\in\pl\Omega$ such that 
\[
\Omega\subset B_R(y) \ \ \AND \ \ x\in\pl B_R(y). 
\]
(To check this, one may choose a sequence $z_j\in\R^N\stm \ol\Omega$ converging to $x$, then choose a sequence $y_j\in Y$ so that $z_j\not\in B_R(y_j)$, and send $j\to\infty$ along a subsequence so that the subsequence converges to a point $y\in\ol Y$. It is clear that $\Omega\subset B_R(y)$ and %$x\in\pl\gO$
$x\in\partial B_R(y)$.)

This is the relationship between existence of solutions and \lq\lq strict convexity\rq\rq\ of the domain.
\begin{theorem}\label{Dirichlet00} Let $\Omega$ be a convex domain.
The Dirichlet problem
\begin{equation} \label{000}%\label{Pbk}
\left\{\begin{array}{cl}
\Ppk(D^2u)=f(x) & \text{in $\Omega$}\\
u=0 & \text{on $\partial\Omega$}
\end{array}\right.
\end{equation}
has a unique solution for any bounded $f\in {\rm C}(\Omega)$ if and only if $\Omega\in {\mathcal C}_{N-k+1}$. 
\end{theorem}
Hence we have a sort of optimal condition for existence. In fact we have better, in the sense that
we prove nonexistence of {\em supersolutions} when the domain is not in ${\mathcal C}_{N-k+1}$.
For the part concerning existence, the construction of supersolutions is given in a constructive and elegant way.
When $k=1$, i.e. when the domain is strictly convex, this result will lead to the construction of the so called eigenfunction corresponding 
to the principal demi-eigenvalue, so generalizing the existence of eigenfunctions provided in \cite{BGI} under the  uniform convexity assumption.\\  
As mentioned above if the forcing term $f$ is positive or at least not too negative near the boundary, 
solutions of \eqref{Pbk} exists as soon as $\Omega$ is convex, strict convexity in order to allow $f$ to be 
negative at the boundary. So the real question is to obtain existence e.g. for $f\equiv -1$.

Interestingly, the presence of the first order term changes dramatically the dependence of the existence of 
solutions on the convexity of the domain. In fact it worsens the situation in the sense that \lq\lq strict convexity\rq\rq\
in general is not enough for existence in the presence of the first order term.
In fact the problem can be of  \lq\lq local\rq\rq\ type, i.e. if there is a point $P$ of the boundary where the principal curvatures are zero, even if the domain is strictly convex, then, for $b>0$ there are no positive supersolutions of
\begin{equation}\label{001}\Ppk(D^2u)+b|D u|=-1\end{equation}
which are zero at that point $P$, see Theorem \ref{bb}.

Or the problem can be of a global nature, i.e. if $\Omega$ is too large, independently of its shape, there are no solutions. More
precisely, if $B_R\subset\Omega$ and $bR\geq k$
there are no supersolutions of \eqref{001}.
Other cases with nonconstant $b$ are also considered in Section \ref{radial weight}.

Due to the relevance of the condition ${\mathcal C}_{j}$, we now give a characterization in term of flatness of the 
boundary,  which will play a role in the proof of Theorem 1.

Given a bounded convex domain $\Omega$ and $x\in\pl\Omega$, we consider the maximal dimension $d_x(\Omega)$ of  linear subspaces $V$ of the tangent space 
of $\pl\Omega$ at $x$ such that 
$(x+V)\cap \pl\Omega$ is a neighborhood of $x$ in the relative topology of $x+V$.  
That is, $d_x(\Omega)$ is the maximum of $m\in\{0,1,\ldots,N-1\}$ such that 
there exist an $m$-dimensional linear subspace $V$ in $\R^N$ and $\delta>0$ 
such that $x+V\cap B_\delta \subset\pl\Omega$. 
We set $d(\Omega)=\max_{x\in\pl\Omega}d_x(\Omega)$. 

\begin{theorem}\label{thm-Cj} Let $\Omega$ be a bounded convex domain. We have 
$\Omega\in\cC_j$ if and only if $d(\Omega)\leq N-j$. 
\end{theorem} 

Finally we wish to somehow compare our results with some results of Blanc and Rossi.
In \cite{BR} they consider the problem
$$\left\{\begin{array}{cc}
\lambda_j(D^2u)=0 &\mbox{in}\ \Omega\\
u=g &\mbox{on}\ \partial\Omega
\end{array}\right.
$$
and they prove that 
if $\Omega\in{\mathcal G}_j\cap {\mathcal G}_{N-j}$ then the above Dirichlet problem is solvable for any $g$
while, if $\Omega$ is not in ${\mathcal G}_j\cap {\mathcal G}_{N-j}$ 
then there %may 
should be some $g$ for which the 
problem is not solvable. The precise definition of ${\mathcal G}_j$ is %given 
%\red{
recalled %} 
in the last section. Let us mention that these
operators, as well as the truncated Laplacians treated here, are fully nonlinear operators and hence it is not possible
to pass immediately from a Dirichlet problem with homogeneous boundary data to a Dirichlet problem with homogeneous forcing term. Nonetheless it is clear that both problems are related.

The definition of these ${\mathcal G}_j$ domains is different from the way we describe the \lq\lq strict convexity\rq\rq\ of 
 our domains. In the sense that we use domains that are  intersection of rotations and translations of \lq\lq($N-j+1$)-dimensional cylinders\rq\rq\ in ${\mathcal C}_{N-j+1}$.

In fact these notions are in general different since ${\mathcal G}_j\cap {\mathcal G}_{N-j}$ contains domain
that may not even be convex. On the other hand, if the domain is convex then the two notions are equivalent as it 
is proved in the last section together with the proof of Theorem \ref{thm-Cj}.

\section{Dirichlet problem} \label{existence}
\subsection{Nonexistence}
We begin by proving that convexity alone is not enough to solve Dirichlet problems for $\Ppk$ 
even for very regular forcing term.
\begin{proposition}\label{non existence}
Let $\Omega\subset \mathbb R^N$ be a convex domain and assume that up to a rigid 
motion there exists  $\delta>0$ such that the $k$-dimensional ball
\begin{equation}\label{flat}
%B_{k,\delta}(0)=\{x\in\mathbb R^N\,:\,x=\sum_{i=1}^{k} t_i e_i, \ |x|<\delta\}\subset 
B_{k,\delta}=\left\{x=(x_1,\ldots,x_k,0,\ldots,0)\in\mathbb R^N\;:\;|x|<\delta\right\}\subset
\partial\Omega.\end{equation}
Then there are no supersolutions $u\in \operatorname{LSC}(\overline\Omega)$ of
\begin{equation}\label{Pb13}
\left\{\begin{array}{cl}
\Ppk(D^2u)=-1 & \text{in $\Omega$}\\
u=0 & \text{on $\partial\Omega$}
\end{array}\right.
\end{equation}
such that 
\begin {equation}\label{H3}
\lim_{x\to0}u(x)=0.
\end{equation}
\end{proposition} 

We note that condition \erf{flat} implies that $d_0(\Omega)\geq k$. We recall that, for $x\in\partial\Omega$, $d_x(\Omega)$ is defined by 
\begin{equation}
\begin{split}
d_x(\Omega)&=\max\left\{m\in\left\{1,\ldots,N-1\right\}\,:\,\exists V\;\text{$m$-dimensional linear subspace on $\RN$}\right.\\
& \hspace{2cm}\left.\text{and $\delta>0$ s.t. $x+V\cap B_\delta\subset\partial\Omega$}\right\}.
\end{split}
\end{equation}

\begin{proof} 
Let us suppose by contradiction that there exists a supersolution $u$ of \eqref{Pb13} satisfying \eqref{H3}. It cannot achieve the  minimum at an interior point $x$, since otherwise we would have $\Ppk(D^2u(x))\geq0$. Hence $u$ is positive in $\Omega$.
% It is clear by the strong minimum principle (see \cite[Remark 2.6]{BGI}) that $u$ is strictly positive in $\Omega$. 
%\\
In view of \eqref{H3}, there exists a positive number $r$ smaller than $\delta$ such that
\begin{equation}\label{equ3+11}
u(x)<\frac{\delta^2}{16k}\quad\text{ for any }x\in B_r\cap\Omega.  
\end{equation}

{\bf Claim:} There exists a point $z\in\Omega$ and $\varepsilon<\frac{\delta}{2}$ such that 
$z\in \{0\}\tim\R^{N-k}\subset\R^N$, $|z|<r$ and the cylinder 
$$
C=\left\{x\in\RN\;:\;\sum_{i=1}^kx_i^2<\frac{\delta^2}{4},\;
\sum_{i=k+1}^{N}(x_i-z_i)^2<\varepsilon^2\right\}\subset\Omega.$$

We suppose that the claim is proved and we go on with the rest of the proof.

Since $z\in B_r$, \eqref{equ3+11} yields   
\begin{equation}\label{equ3}
u(z)<\frac{\delta^2}{16k}. 
\end{equation}
Let \begin{equation}\label{phi}
\varphi(x)=-\alpha\sum_{i=1}^kx_i^2-\beta\left[\sum_{i=k+1}^{N}(x_i-z_i)^2-\varepsilon^2\right],
\end{equation}
where 
\begin{equation}\label{equ1}
\alpha=\frac{8u(z)}{\delta^2},\quad\beta=\frac{2u(z)}{\varepsilon^2}.
\end{equation}

We claim that $\displaystyle\min_{\overline C}(u-\varphi)$ is attained at some point $\xi\in C$.\\ Let $x\in\partial C$. If $\sum_{i=1}^kx_i^2=\frac{\delta^2}{4}$ then
$$
u(x)-\varphi(x)\geq-\varphi(x)\geq\alpha\frac{\delta^2}{4}-\beta\varepsilon^2=0
$$
in view of \eqref{equ1}.\\ Otherwise $\sum_{i=k+1}^{N}(x_i-z_i)^2=\varepsilon^2$ and 
$$
u(x)-\varphi(x)\geq-\varphi(x)=\alpha\sum_{i=1}^kx_i^2\geq0.
$$
Since
$$
u(z)-\varphi(z)=u(z)-\beta\varepsilon^2<0,
$$
then necessarily $u-\varphi$ has a minimum at an interior point, say $\xi\in C$, and 
\begin{equation}\label{equ2}
\Ppk(D^2\varphi(\xi))\leq-1.
\end{equation}
On the other hand $$D^2\varphi={\rm diag}(\underbrace{-2\beta,\ldots,-2\beta}_{N-k\;\text{times}},\underbrace{-2\alpha,\ldots,-2\alpha}_{k\;\text{times}})$$
with $\alpha<\beta$. Then using \eqref{equ3} and \eqref{equ1} one has 
$$
\Ppk(D^2\varphi(\xi))=-2\alpha k>-1
$$
in contradiction to \eqref{equ2}.

We now give the proof of the claim.
Since the origin is on $\pl\Omega$, we may choose a $y\in\Omega$ so that
$|y|<r$.  
Set 
\[y^{(1)}=(y_1,\ldots,y_k,0,\ldots,0), \quad y^{(2)}=(0,\ldots,0,y_{k+1},\ldots,y_N)\in\R^N.
\]
By assumption \eqref{flat} $-y^{(1)}\in B_{k,\delta}$ and  using the convexity of $\Omega$  

\[
\frac 12 y^{(2)}=\frac 12 y -\frac 12 y^{(1)}\in \frac 12 \Omega+\frac 12\pl\Omega\subset \Omega. 
\]
Set $2z=\frac 12 y^{(2)}$ and note that $2z\in \{0\}\tim\R^{N-k}\subset\R^N$.  
Select a positive constant $\varepsilon<\frac\delta2$ so that $B_{2\varepsilon}(2z)\subset \Omega$ and note 
that 
\[
B_{k,\frac \delta 2}+B_{\varepsilon }(z)= \frac 12 B_{k,\delta}+\frac 12 B_{2\varepsilon}(2z)\subset \Omega.
\]
Then we have the inclusion for the cylinder 
$$
C=\left\{x\in\RN\;:\;\sum_{i=1}^kx_i^2<\frac{\delta^2}{4},\;\sum_{i=k+1}^{N}
(x_i-z_i)^2<\varepsilon^2\right\}\subset\Omega.
$$
\end{proof}
%\color{black}
%%%%%%%%%%%%%%%%%%%%%%%%%%%%%%%%

\subsection{Existence}
In order to solve the Dirichlet problem with general right hand side $f$ we should impose that $\partial\Omega$ has at least $N-k$ directions of strict convexity, as anticipated in the Introduction, see \eqref{j strict convex}-\eqref{j strict convex 1}.

\begin{theorem}\label{Dirichlet1}
Let $\Omega\in {\mathcal C}_{N-k+1}$  and let $f\in {\rm C}(\Omega)$ be bounded. Then the Dirichlet problem
\begin{equation}\label{Pbk}
\left\{\begin{array}{cl}
\Ppk(D^2u)=f(x) & \text{in $\Omega$}\\
u=0 & \text{on $\partial\Omega$}
\end{array}\right.
\end{equation}
has a unique solution.
\end{theorem}

\def\IN{\text{ in }}
Before discussing the Dirichlet problem \erf{Pbk}, for a basis of our discussion, 
we state a proposition concerning the comparison principle.

%\begin{color}{red}
\begin{proposition}\label{g-comp} Let $\Omega\subset\R^N$ be a bounded domain 
and $b, f\in {\rm C}(\Omega)$. Let $F$ denote either $\cP_k^+$ of $\cP_k^-$. 
Let 
$v\in\USC(\ol\Omega)$ and $w\in\LSC(\ol\Omega)$ be a sub and supersolution 
of    
\begin{equation}\label{g-comp1}
F(D^2u)+b(x)|Du|=f(x) \ \ \IN \Omega
\eeq
and satisfy $v\leq w$ on $\pl\Omega$. Moreover, assume that either of
$b$, $v$ or $w$ is locally Lipschitz in $\Omega$. 
Then, under one of the following conditions, we have $v\leq w$ in $\Omega$. 
\begin{enumerate}
\item[\emph{(i)}]There exists a ball $B_R$ such that $\Omega\subset B_R$ 
and that $\|b\|_\infty R\leq k$.  
\item[\emph{(ii)}] $f>0$ in $\Omega$ or $f<0$ in $\Omega$. 
\end{enumerate}
\end{proposition}

A comparison theorem under the condition (i) above (without equality) can be found in \cite[Proposition 4.1]{GV}, where it is also shown by a counterexample that the assumption $\|b\|_\infty R\leq k$ cannot be improved in general.

It should be noted that $\USC(X)$ (resp.,  $\LSC(X)$) denotes here the set 
of \emph{real-valued} upper (resp., lower ) semicontinuous functions on $X$.

\bproof[Outline of proof] We consider only the case $F=\cP^+_k$. Fix a small $\varepsilon>0$ 
and consider the function $v_\varepsilon=v-\varepsilon$, which is still a subsolution of \erf{g-comp1}. Since $v_\varepsilon<w$ on $\pl\Omega$ and $v_\varepsilon-w\in\USC(\ol\Omega)$, 
there exists $\delta\in(0,\,\varepsilon)$ so that for $\Omega_\delta=\{x\in\Omega\mid \dist(x,\pl\Omega)>\delta\}$, 
we have $v_\varepsilon<w$ on $\pl\Omega_\delta$.  Note that either $b$, $v_\varepsilon$ or $w$ is Lipschitz continuous in $\Omega_\delta$. 

The next step is to replace either 
$v_\varepsilon$ or $w$ by its small modification, which is, respectively , a strict subsolution 
or strict supersolution of \erf{g-comp1} in $\Omega_\delta$.  

Let $0<\gamma<1$ and first consider the case (i). 
By translation, we may assume that $\Omega_\delta\subset B_r$ for some $0<r<R$   
and consider the function $v_{\varepsilon,\gamma}(x):=v_\varepsilon(x)+\gamma|x|^2/2$ 
with $\gamma>0$. 
This function $v_{\varepsilon,\gamma}$ is a subsolution of 
\[
\cP_k^+(D^2u-\gamma I)+b(x)|Du-\gamma x|=f(x) \ \ \IN\Omega,
\]
where $I$ denotes the $N\tim N$ unit matrix. From this, it is easily seen
  $v_{\varepsilon,\gamma}$ is a subsolution of 
\[
\Ppk(D^2u)+b(x)|Du|=f(x) +\gamma(k-\|b\|_\infty |x|) \ \ \IN\Omega.
\]
Note that, since $\gamma(k-\|b\|_\infty|x|)>0$, 
$v_{\varepsilon,\gamma}$ is a strict subsolution of \erf{g-comp1} in $\Omega_\delta$ and that 
$v_{\varepsilon,\gamma}<w$ on $\pl\Omega_\delta$ for $\gamma$ sufficiently small. 

Next, consider the case (ii). If $f>0$ in $\Omega$, then, by the homogeneity of 
the operator $F(D^2\cdot)+b|\nabla \cdot |$, the function
$v_{\varepsilon,\gamma}=(1+\gamma)v_\varepsilon$ is a subsolution of 
\[
\Ppk(D^2u)+b|Du|=(1+\gamma)f \ \ \IN \Omega_{\delta},
\]
which means that $v_{\varepsilon,\gamma}$ is a strict subsolution of \erf{g-comp1} in $\Omega_\delta$. Similarly, if $f<0$, the function $v_{\varepsilon,\gamma}=(1-\gamma)v_\varepsilon$ is a strict 
subsolution of \erf{g-comp1} in $\Omega_\delta$.  We may take 
$\gamma>0$ small enough so that $v_{\varepsilon,\gamma}\leq w$ on $\pl\Omega_\delta$

We may now apply \cite[Theorem 3.3 and Sections 5.A, 5.C]{CIL}, to conclude that
$v_{\varepsilon,\gamma}\leq w$ in $\Omega_\delta$. Sending $\gamma\to 0$ first and then 
$\varepsilon\to 0$ complete the proof.  

Here are two remarks. 
For use of \cite[Section 5.A]{CIL}, we observe 
that, if we set $G(x,p,X)=-\Ppk(X)-b(x)|p|$ and two $N\tim N$ matrices satisfy 
\begin{equation}\label{matrix-ineq}
\bmat X&0\\ 0 &-Y\emat\leq 
3\alpha\bmat I&-I\\-I&I\emat \ \ \text{ for some }\alpha>0,
\eeq
then we have $X\leq Y$ and therefore 
\[\bald
G(y,\alpha(x-y),Y)-G(x,\alpha(x-y),X) &\,\leq G(y,\alpha(x-y),Y)-G(x,\alpha(x-y),Y)
\\&\,\leq \alpha |b(x)-b(y)||x-y| \ \ \text{ for } x,y\in\Omega_\delta.
\eald
\]
This shows that, taking limit under the condition that $X$ and $Y$ satisfy \erf{matrix-ineq} and $\alpha|x-y|\leq C$ for a fixed constant $C>0$, we have 
\[
\limsup_{|x-y| \to 0}\left[G(y,\alpha(x-y),Y)-G(x,\alpha(x-y),X)\right]\leq 0. 
\]
This observation is not enough for a direct application of \cite[Section 5.A]{CIL}, but, 
in fact, a slight modification of the argument in \cite[Section 5.A]{CIL}
yields $v_\varepsilon\leq w$ in $\Omega_\delta$ when either $v$ and $w$ is in $\Lip(\Omega_\delta)$. 

Secondly, it is not trivial to see in the case  of (ii) that if $\gamma>0$ is small enough, then $v_{\varepsilon,\gamma}\leq w$ on $\pl\Omega_{\delta}$.  In fact, since $v_\varepsilon, -w\in\USC(\ol\Omega_\delta)$, we infer that $\max_{\pl\Omega_\delta}(v_\varepsilon-w)<0$. 
Also, by the semicontinuity, there 
is a constant $M>0$ such that $v_\varepsilon, -w\leq M$ on $\pl\Omega_\delta$. 
For $x\in\pl\Omega_\delta$ and $\gamma>0$ sufficiently small, if $v_\varepsilon(x)\leq -2M$, then 
\[
v_{\varepsilon,\gamma}(x)=(1\pm \gamma)v_\varepsilon(x)\leq -2(1\pm\gamma)M\leq -M\leq w(x),
\]
and otherwise, we have $-2M<v_{\varepsilon}(x)\leq M$ and 
\[
v_{\varepsilon,\gamma}(x)\leq v_{\varepsilon}(x)+\gamma |v_\varepsilon(x)|< v_\varepsilon(x)+2\gamma M\leq w(x). 
\]
This way, one gets $v_{\varepsilon,\gamma}\leq w$ on $\pl\Omega_\delta$ for small $\gamma>0$. 
\eproof
%\end{color}

The proof of Theorem \ref{Dirichlet1} is carried out by means of Perron method. It is worth pointing out that the 
standard procedure to construct  subsolutions which are null on $\partial\Omega$ (see e.g. \cite[Section 9]{Cprimer}) 
works for  $\Ppk$ which is in fact a sup operator. 
%\red{QUESTION: what does the following mean ? ``which is in fact a sup operator''}
On the other hand it fails for supersolutions owing to the strong 
degeneracy of $\Ppk$ with respect to inf-type operations. The geometry of $\Omega$  plays here a crucial role.
  
Hence we will start by recalling a property concerning strict convex domains.
Let $\Omega$ be a convex domain of $\R^N$ and $z\in\pl\Omega$. 
The set $N(z)=N_\Omega(z)$ of outward normal unit vectors at $z$ 
is defined by
\[
N(z)=\{p\in\R^N\mid |p|=1,\ \  (x-z)\cdot p\leq 0 \ \text{ for all }x\in\Omega\}.
\]  
It is well-known (a consequence of the Hahn-Banach theorem) that $N(z)\not=\emptyset$ for every $z\in\pl\Omega$. 

\begin{definition} A domain $\Omega\subset\R^N$ is strictly convex if 
\[ 
(1-t)x+ty\in\Omega\ \ \ \text{ for all } \ 
x,y\in\ol\Omega, \text{ with } x\not=y, \ 0<t<1.
\]
\end{definition}

\begin{lemma}\label{lemma strict convexity}

 If a domain $\Omega$ is strictly convex, then 
\[
N(x)\cap N(y)=\emptyset \ \ \ \text{ for all }\ x,y\in\pl\Omega, \text{ with } x\not=y.
\]
\end{lemma}
For the convenience of the reader, we give a proof.
\bproof Let $x,y\in\pl\Omega, \ x\not=y$. Suppose that there is $p\in N(x)\cap N(y)$. It follows that 
\[
(z-x)\cdot p\leq 0,\quad (z-y)\cdot p\leq 0 \ \ \ \forall z\in\ol\Omega.
\]
Adding these two yields
\begin{equation}\label{eq1lemma}
\left(z-\frac{x+y}{2}\right)\cdot p\leq 0 \ \ \forall z\in\ol\Omega. 
\end{equation}
Since $\Omega$ is strictly convex, we have 
\[
\frac{x+y}2\in\Omega,
\]
and, therefore, there exists $\delta>0$ such that 
\[
B_{\delta}\Big(\frac{x+y}{2}\Big)\subset \Omega,
\]
and, in particular, 
\[
z:=\frac{x+y}{2}+\delta p\in\ol\Omega,
\]
which shows that
\[
\left(z-\frac{x+y}{2}\right)\cdot p=\delta|p|^2=\delta>0,
\]
contradicting \eqref{eq1lemma}. 
\eproof

Let $\omega$ be a bounded strictly convex domain in $\R^j$.  
Let $\cF$ be the collection of functions
\[
f(x)=\frac{1}{2}(R^2-|x-x_0|^2),
\]
where $R>0$, $x,\,x_0\in\R^j$, and, moreover,
\[
f(x)> 0 \ \ \ \text{ on }\omega.
\]
It is clear that $\cF\not=\emptyset$. We set 
\begin{equation}\label{supersolution}
\psi(x)=\inf_{f\in\cF} f(x) \ \ \ \text{ for }\ x\in\ol\omega.
\end{equation}

It is clear that $\psi$ is concave, since it is infimum of concave functions. Hence $\psi\in\rm{Lip_{loc}}(\omega)$ and one has $\Ppo(D^2\psi)\leq-1$ in $\omega$. Moreover
\[
\psi\in\USC(\overline\omega),\qquad\qquad  \psi\geq 0 \ \ \text{ on }\ \ol\omega.
\]

\bthm\label{supersolution1}
\label{psi} $\psi(x)=0$ for all $x\in\pl\omega$.  In particular, $\psi\in {\rm C}(\ol\omega)$. 
\ethm

\bproof Fix $z_0\in\pl\omega$ and $p\in N(z_0)$. 
By rotation and translation, we may 
assume that $z_0=0$ and $p=(0,\ldots,0,1)$. For generic $z\in\R^j$, we write
\[
z=(x,y), \quad x\in\R^{j-1},\, y\in\R. 
\]  
We choose $R_0>0$ so that 
\[
\omega\subset \{z=(x,y)\in\R^j\mid |x|^2+y^2<R_0^2\}.
\]
For any $R\geq R_0$, set 
\[
\rho\equiv \rho(R):=\sup\{h\geq 0\mid \omega\subset\{(x,y)\mid |x|^2+(y+h)^2<R^2\}\}. 
\]
It is clear by simple geometry that $0<\rho<\infty$, $R\mapsto\rho(R)$ is increasing and 
\[
\lim_{R\to\infty}\rho(R)=\infty. 
\]
Indeed, since, for $R\geq R_0$, 
\[
B_{R}\big((0,-(R-R_0))\big)\supset B_{R_0}((0,0))\supset\omega,
\]
we see that
\[
\rho(R)\geq R-R_0,
\]
which shows that 
\[
\lim_{R\to\infty}\rho(R)=\infty. 
\]

Note also that
\[
\ol\omega\subset \{(x,y)\mid |x|^2+(y+\rho(R))^2\leq R^2\},
\]
which implies that the function
\[
f_R(x,y):=\frac{1}{2}\left(R^2-|x|^2-(y+\rho(R))^2\right)
\]
is positive in $\omega$, that is, $f_R\in\cF$.  

Observe that, if $r>R$, then 
\[
(0,0)\not\in \ol B_R((0,-r)),
\]
which implies that 
\[
\ol\omega\not\subset \ol B_R((0,-r)),
\]
and hence
\[
\rho(R)\leq r, \quad\text{ and, consequently, } \quad \rho(R)\leq R.
\]

We need only to show that 
\[
\lim_{R\to\infty}f_R(0)=0. 
\]
(Notice that this implies that $\psi(0)\leq 0$ and, moreover, that 
$\displaystyle\limsup_{\ol\omega\ni x\to 0}\psi(x)\leq \psi(0)\leq 0$ while $\displaystyle\liminf_{\ol\omega\ni x\to 0}\psi(x)\geq 0$ since $\psi\geq 0$ in $\ol\omega$.)

By the definition of $\rho(R)$ and the compactness of $\ol\omega$, 
there exists a point $z_R=(x_R,y_R)\in\pl\omega$ such that $f_R(z_R)=0$. That is, 
\[
|x_R|^2+(y_R+\rho(R))^2=R^2. 
\]

By simple geometry again, we see that
\[
p_R\equiv (\alpha_R,\beta_R):=\frac{1}{\sqrt{|x_R|^2+|y_R+\rho(R)|^2}}(x_R,y_R+\rho(R))\in N(z_R). 
\]
By the compactness of $\pl\omega$, there is a sequence $R_j\to\infty$ 
such that 
\[
z_{R_j} \to z_\infty\in\pl\omega.
\]
Observe that, since $\lim_{R\to\infty}\rho(R)=\infty$,  
\[
p_{R_j}\to (0,1). 
\]
%Since $z\mapsto N(z)$ is upper continuous on $\pl\Omega$,
Passing to the limit in the inequality $(x-z_{R_j})\cdot p_{R_j}\leq0$\ for all $x\in\omega$, we see that 
\[
(0,1)\in N(z_\infty).
\]
However, since $(0,1)\in N(0)=N((0,0))$, by the strict convexity of $\omega$ 
(Lemma \ref{lemma strict convexity}), we must have 
\[
z_\infty=0. 
\]
The above argument implies that 
\[
\lim_{R\to \infty}z_R=0. 
\]

Observe that, since $R-R_0\leq \rho(R)\leq R$,
\begin{equation}\label{2}
\left|\frac{R\alpha_R\cdot x_R}{\beta_R}\right|=\left|\frac{Rx_R\cdot x_R}{y_R+\rho(R)}\right|
=\left|\frac{R}{y_R+\rho(R)}\right||x_R|^2\ \to 0\quad\text{ as }R\to\infty. 
\end{equation}
Noting that $p_R$ is an outward normal vector to $B_R((0,-\rho(R))$ at $z_R$
and that $(0,R-\rho(R))\in\pl B_R((0,-\rho(R))$, we have
\[
0\geq p_R\cdot((0,R-\rho(R))-z_R)=\beta_R(R-\rho(R))-\alpha_R\cdot x_R-\beta_R y_R,
\] 
and, if $\beta_R>0$, then
\[
R-\rho(R)\leq \frac{\alpha_R\cdot x_R}{\beta_R}+y_R.
\]
Since $(0,1)\in N((0,0))$, we have
\[
0\geq (0,1)\cdot ((x_R,y_R)-(0,0))=y_R. 
\]
Thus, if $\beta_R>0$, then
\[
R-\rho(R)\leq \frac{\alpha_R\cdot x_R}{\beta_R}.
\]
Combining this with \eqref{2}, we see that, as $R\to\infty$,
\[
0\leq R(R-\rho(R))\leq R\frac{\alpha_R\cdot x_R}{\beta_R} \to 0,
\]
and, moreover,
\[
\lim_{R\to\infty}(R+\rho(R))(R-\rho(R))=0,
\]
Hence, 
\[
\lim_{R\to\infty}f_R(0,0)=\frac 12 \lim_{R\to\infty}(R^2-\rho(R)^2)=0. \qedhere
\]
\eproof 

\begin{lemma}\label{funPsi}Let $\psi$ be the function defined by \eqref{supersolution} and $\omega\subset\R^j$ be as in \eqref{supersolution}.  
 Assume that $j\leq N$ and set $C=\omega\tim\R^{N-j}$. 
Define the function $\Psi$ on $\ol C=\ol\omega \tim\R^{N-j}$ by
$\Psi(x)=\psi(x_1,\ldots,x_j).$ Let $k\in\N$ be such that $N-j<k$. 
Then, $\Psi$ is continuous on $\ol C$ and $\cP_k^+(D^2\Psi)\leq -1$ in $\omega\tim\R^{N-j}$. 
\end{lemma}

By definition, a set $C\subset \R^N$ is in $C_j$ if and only if 
$C=\omega\tim\R^{N-j}$ for some bounded strictly convex $\omega \subset\R^j$.
The function $\psi$ depends only on $\omega$ and if $C=\omega\tim \R^{N-j} 
\in C_j$, then the function $\Psi$,  defined in the lemma above,  is considered to depend only on $C$. Thus, for later reference, we write $\Psi_C$ for this $\Psi$.   

\bproof  The continuity of $\Psi$ is obvious, since $\psi\in {\rm C}(\ol\omega)$.  
Recalling \eqref{supersolution}, the function $\Psi$ is given as 
the infimum of a family of functions $f$ on $\R^n$ of the form  
\[
f(x)=\frac 12(R^2-\sum_{i=1}^j (x_i-x_{0,i})^2), 
\]
for some $R>0$ and $x_0\in\R^N$. Observe that 
\[
D^2f={\rm diag}(\underbrace{-1,\ldots,-1}_{j\;\text{times}},\underbrace{0,\ldots,0}_{N-j\;\text{times}}),
\]
and, since $k>N-j$, $\cP_k^+(D^2f)\leq -1$ in $\omega\tim\R^{N-j}$.  By the stability 
of the supersolution property under inf-operation, we conclude that 
$\cP_k^+(D^2\Psi)\leq -1$. 
\eproof

\begin{proof}[Proof of Theorem \ref{Dirichlet1}]
Since $\Omega$ is a convex set,  the uniform exterior sphere condition is satisfied. Then for $r=|x|$ let us consider the function  $G(r)=r^{-\alpha}-1$ where $\alpha=\max\left\{k-1,1\right\}$. Observe that for $r>1$
$$G(r)<0,\ G^\prime(r)= -\alpha r^{-(\alpha+1)}<0,\ G^{\prime\prime}(r)= \alpha(\alpha+1)r^{-(\alpha+2)}>0.$$
Let $\displaystyle\Phi(x)=\sup_{x_b\in\partial\Omega}G(|x-z_b|)$, where $z_b$ is such that $|x_b-z_b|=1$ and for  any $x\in\Omega$ one has  $|x-z_b|>1$. Then $\Phi\in {\rm C}(\overline\Omega)$ and $\Phi=0$ on $\partial\Omega$. Moreover $$\Ppk(D^2\Phi(x))\geq \frac{1}{(1+{\rm diam}(\Omega))^{\alpha+2}}.$$
Then $\underline u=M_1 \Phi$, with $M_1=M_1(\Omega,\alpha,\left\|f\right\|_\infty)$ sufficiently large, is a continuous subsolution of \eqref{Pbk} which vanishes on $\partial\Omega$.

Now we provide a continuous supersolution $\overline u$ such that $\overline u=0$ on $\partial\Omega$. 
%Since $\Omega\in{\mathcal C}_{N-k+1}$, it satisfies \eqref{j strict convex 2}. 
By the definition of ${\mathcal C}_{N-k+1}$, since $\Omega\in{\mathcal C}_{N-k+1}$,
the set $S_{N-k+1}$ is given associated with $\Omega$. 
In view of \erf{j strict convex}, define for $x\in \ol\Omega$
$$
w(x)=\inf_{(O,C)\in S_{N-k+1}}\Psi_C\left(O^Tx\right),
$$
where $\Psi_C$ is the function on $\ol C$ defined in Lemma \ref{funPsi} 
(see also a comment after the lemma).  From the properties of 
the function $\psi$ defined by \eqref{supersolution} it follows 
that $\Psi_C$ is concave 
and nonnegative in $C$.  
Theorem \ref{psi} ensures that $\Psi_C=0$ on $\pl C$ and $\Psi_C\in {\rm C}(\ol C)$. 
It is now obvious that $w$ is nonnegative, concave and upper semicontinuous  
on $\ol\Omega$ and that $\Psi_C(O^Tx)=0$ for $x\in O(\pl C)=\pl(OC)$.  It follows from \eqref{j strict convex 1} that 
\[
\pl \Omega\subset \bigcup_{(O, C)\in S_{N-k+1}}\pl (OC),
\]
which implies that $w\leq 0$ on $\pl\Omega$. These properties of $w$ 
guarantee that $w\in {\rm C}(\ol\Omega)$ and $w=0$ on $\pl\Omega$. 

Noting that if we set $j=N-k+1$, then $N-j<k$,
we see by Lemma \ref{funPsi} that for any $(O,C)\in S_{N-k+1}$, 
$\cP_k^+(D^2\Psi_C)\leq -1$ in $C$ and moreover,
by the invariance of the operator $\cP_k^+$ under orthogonal transformation, 
that the function $v(x):=\Psi_C(O^Tx)$
satisfies $\cP_k^+(D^2v)\leq -1$ in $OC$. The stability of the subsolution 
property under inf-operation implies that $\cP_k^+(D^2v)\leq -1$ in $\Omega$. 
We set $\ol u=\|f\|_\infty w$ and note that $\ol u\in {\rm C}(\ol\Omega)$, 
$\ol u=0$ on $\pl\Omega$ and $\cP_k^+(D^2\ol u)\leq f$ in $\Omega$. 

Now, the Perron method yields a function $u$ on $\ol\Omega$ such that the upper 
semicontinuous envelope $u^*$ of $u$ is a subsolution of  $\cP_k^+(D^2u)=f$ in $\Omega$, the lower semicontinuous envelope $u_*$ of $u$ 
is a supersolution of $\cP_k^+(D^2u)=f$ in $\Omega$, 
and $\underline{u}\leq u_*\leq u\leq u^*\leq \ol u$ on $\ol\Omega$. 
The standard argument including comparison between $u^*$ and $u_*$ 
assures that $u\in {\rm C}(\ol\Omega)$ and $u$ is a solution of \eqref{Pbk}.
\end{proof}

\begin{proof}[Proof of Theorem \ref{Dirichlet00}]
%The sufficient condition 
Sufficiency of the $\cC_{N-k+1}$ property of $\Omega$ 
has been proved in Theorem \ref{Dirichlet1}. 
Its necessity follows from Proposition \ref{non existence}. Indeed, if $\Omega$ is not in 
$\cC_{N-k+1}$, then, by Theorem \ref{thm-Cj}, $d(\Omega)\geq k$, which means 
after translation and orthogonal transformation that $0\in\pl\Omega$, $d_0(\Omega)\geq k$, 
and, moreover, condition \erf{flat} holds. Thus, Proposition \ref{non existence} 
implies that problem \erf{000}, with $f=-1$, does not have a solution continuous 
up to 
the boundary $\pl\Omega$. 
\eproof 

%%%%%%%%%%%%%%%%
\subsection{Application: eigenfunctions for $\Ppo$ in strictly convex domains}$\,$

Following the Berestycki-Nirenberg-Varadhan approach concerning the validity of the Maximum Principle, see \cite{BNV}, we have defined in \cite{BGI} as candidate for the principal eigenvalue the values
$$
\mu^+_k=\sup\left\{\mu\in\R\;:\;\exists \varphi\in \operatorname{LSC}(\Omega),\;\varphi>0,\;\Ppk(D^2\varphi)+\mu\varphi\leq0\;\,\text{in $\Omega$} \right\},
$$
$$
\mu^-_k=\sup\left\{\mu\in\R\;:\;\exists \varphi\in \operatorname{LSC}(\Omega),\;\varphi<0,\;\Ppk(D^2\varphi)+\mu\varphi\geq0\;\,\text{in $\Omega$} \right\}.
$$
For the convenience of the reader it is worth pointing out the change of notation: here $\mu^+_k$ corresponds to 
what  in \cite{BGI} was called $\mu^-_k$ and vice versa, since in the present paper we deal with the maximal 
operator $\Ppk$, whereas in \cite{BGI} we considered the minimal one $\Pmk$. 
In particular we proved that $\mu^-_k=+\infty$ while $\mu^+_k<\infty$, so we will concentrate on the latter.\\
Even in the degenerate framework of the operators $\Ppk$, we showed that if $\Omega$ is uniformly convex, 
then   $\mu^+_k$ gives  threshold for  the Maximum Principle (see \cite[Theorems 4.1, 4.4]{BGI}), this is true 
also for more general equations depending on gradient terms. Moreover, when $k=1$, there exists a positive 
principal eigenfunction.\\
One of the question raised in \cite{BGI} concerned the necessity of the uniform convexity of the domain. In the
next theorem we show that the strict convexity assumption of $\Omega$ is sufficient for the existence of a principal eigenfunction, at least when there are no first order terms.

\begin{theorem}\label{exi}
Let  $\Omega$ be a bounded strictly convex domain and let $f$ be a continuous and bounded function in $\Omega$.
Then there exists a  solution $u\in{\rm C}(\overline\Omega)$ of
\begin{equation}\label{dir1}
\left\{\begin{array}{cl}
\Ppo(D^2u)+\mu u=f &\mbox{in}\ \Omega\\
u=0 &\mbox{on}\ \partial\Omega
\end{array}
\right.
\end{equation}
in the following two cases:
\begin{itemize}
	\item for  $\mu< \mu_1^+$;
	\item for any $\mu$ if $f\geq0$.
\end{itemize}
Moreover in the case $\mu< \mu_1^+$ the solution is unique.
\end{theorem}

%\begin{color}{red}
The uniqueness part of Theorem \ref{exi} is an obvious consequence 
of the following lemma.

\begin{lemma} \label{comp<mu} Under the hypothesis of Theorem 
\ref{exi}, let $\mu<\mu_1^+$, let $u\in\USC(\Omega)$ and 
$v\in\LSC(\Omega)$ be sub and supersolution of
\[
\Ppo(D^2u)+\mu u=f  \ \ \mbox{in}\ \Omega,
\]
respectively, and assume that 
\[
\lim_{\Omega\ni x\to \pl\Omega}(u(x)-v(x))\leq 0. 
\]
Then, $u\leq v$ in $\Omega$. 
\end{lemma}

\bproof Set $w=u-v$ and observe (see \cite[Lemma 3.1]{GV} and also \cite[Theorem 5.8]{CC}, \cite[Proposition 4.1]{IY}) that $w$ is a subsolution of 
\begin{equation}\label{comp0<mu}
\Ppo(D^2w)+\mu w=0 \ \ \mbox{ in }\Omega.
\eeq
The maximum principle (\cite[Theorem 4.1 and Remark 4.8]{BGI}) yields $w\leq 0$ in $\Omega$, 
which concludes the proof. 
\eproof

For the reader's convenience, we recall here \cite[Proposition 3.2]{BGI} 
stated for $\Ppo$. 

\begin{lemma} \label{Lip-est} Let $u\in \LSC(\ol\Omega)$ be a supersolution of 
\begin{equation}\label{dir2}
\left\{\begin{array}{cl}
\Ppo(D^2u)=f(x) &\mbox{in}\ \Omega\\
u=0 &\mbox{on}\ \partial\Omega
\end{array}
\right.
\end{equation} 
Then for each $\varepsilon>0$ there exists a positive constant $L_\varepsilon=L_\varepsilon(\|u\|_\infty,\|f\|_\infty)$ such that 
\[
|u(x)-u(y)|\leq L_\varepsilon|x-y| \ \ \text{ for }x,y\in\Omega_\varepsilon,
\]
where $\Omega_\varepsilon:=\{x\in\Omega\mid \dist(x,\pl\Omega)>\varepsilon\}$. Furthermore, 
if for some constant $C>0$ 
\[
u(x)\leq C\dist(x,\pl\Omega) \ \ \text{ for } x\in\ol\Omega,
\]
then there exists a positive constant $L=L(C,\|u\|_\infty,\|f\|_{\infty})$ such that
\[
|u(x)-u(y)|\leq L|x-y| \ \ \text{ for }x,y\in\Omega.
\]
\end{lemma} 
%\end{color}

\begin{proof}[Proof of Theorem \ref{exi}] $\,$ We need only to prove 
the existence part of the theorem. 

The Dirichlet problem \erf{dir2} 
is uniquely solvable by means of Theorem \ref{Dirichlet1}. 
We henceforth assume that $\mu>0$. We shall first prove Theorem \ref{exi}
for $f:=h\leq 0$ then for $f:=g\geq 0$ and any $\mu$, and, finally, for the general case. 

\smallskip

Let $h=-f^-\leq 0$.
Let $(w_n)_{n\in\N}\subset {\rm C}(\overline\Omega)$ be the sequence defined in the following way:\\ set $w_1=0$ and, given $w_n$, define $w_{n+1}$ as the unique solution of
\begin{equation}\label{dir3}
\left\{\begin{array}{cl}
\Ppo(D^2w_{n+1})=h-\mu w_n &\mbox{in}\ \Omega\\
w_{n+1}=0 &\mbox{on}\ \partial\Omega.
\end{array}
\right.
\end{equation}
Note that $(w_n)_{n\in\N}$ is nondecreasing, in particular $w_n\geq0$ in $\Omega$ for any $n\in\N$. 
At each step the existence is done by using zero as a subsolution and  $(\|h\|_\infty+\mu\|w_n\|_\infty)\psi$ as a supersolution, where $\psi$ is the function defined by \eqref{supersolution} in the case $\omega=\Omega$, see Theorem \ref{supersolution1}.
We need to prove that the sequence  $(\|w_n\|_\infty)_{n\in\N}$ is bounded.

Suppose that it is not, hence up to some subsequence  $\lim_{n\rightarrow+\infty}\|w_n\|_\infty=+\infty$. Then
consider $v_n=\frac{w_n}{\|w_n\|_\infty}$. Then $\|v_n\|_\infty=1$ and $v_n$ satisfies
$$\Ppo(D^2v_{n+1})=\frac{h}{\|w_{n+1}\|_\infty}-\mu v_n\frac{\|w_n\|_\infty}{\|w_{n+1}\|_\infty}.$$
By construction $v_n$ is a sequence of bounded functions. We want to prove that they are equicontinuous.
Observe that, 
$$\frac{h}{\|w_{n+1}\|_\infty}-\mu v_n\frac{\|w_n\|_\infty}{\|w_{n+1}\|_\infty}\geq -\|h\|_\infty-\mu.$$
Hence $0\leq v_n\leq (\|h\|_\infty+\mu)\psi$ for any $n\in\N$.

For any $\delta>0$, in $\Omega_\delta:=\{x\in \overline\Omega\mid 
\dist(x,\partial\Omega)>\delta\}$, the functions $v_n$ are uniformly Lipschitz continuous by Lemma \ref{Lip-est}.
For any $\varepsilon>0$, choose $\delta>0$ such that $(\|h\|_\infty+\mu)\psi\leq\frac{\varepsilon}{2}$ for any 
$x\in \Omega\stm\Omega_\delta$. 
Hence for any $x$, $y$ in $\Omega\stm \Omega_\delta$:
$$|v_n(x)-v_n(y)|\leq v_n(x)+v_n(y)\leq  (\|h\|_\infty+\mu)(\psi(x)+\psi(y))\leq \varepsilon.$$

Hence the sequence $(v_n)_{n\in\mathbb N}$ is equicontinuous in $\overline\Omega$ and up to a subsequence, for some $k\leq 1$, $v_n$ converges
to $v_\infty$ solution of
$$\Ppo(D^2v_{\infty})+k\mu v_\infty=0, \quad\mbox{in}\ \Omega,\, v_\infty=0 \quad\mbox{on}\ \partial\Omega.$$
By maximum principle, since $k\mu<\mu^+_1$ this implies that $v_\infty=0$. This is a contradiction since $\|v_\infty\|_\infty=1$.%We complete the proof.

 We have just proved that there exists some constant $K$ such that $\|w_n\|_\infty\leq K$
and clearly
\begin{equation}\label{induction}
0\leq w_n \leq (\|h\|_\infty+K\mu)\psi.
\end{equation}
Hence, reasoning as above the sequence is also equicontinuous in $\overline\Omega$ and then, up to a subsequence it converges to
a solution $\overline w$ of \eqref{dir1} with $f$ replaced by $h=-f^-$.

\smallskip

Let us consider now the case $f=f^+$.  As above let us define the sequence 
$(w_n)_{n\in\N}$ by setting 
$w_1=0$ and, once $w_n$ is given , solving \eqref{dir3} with $f^+$ in place of $h$. 
In particular $w_{n+1}\leq w_n\leq0$. 
%\begin{color}{red}
Arguing by contradiction as above and 
applying the global Lipschitz regularity result
(Lemma \ref{Lip-est}) to 
negative functions  
$v_n:=w_n/\|w_n\|_\infty$, 
we observe that the sequence $(w_n)_{n\in\N}$ is bounded in ${\rm C}(\ol\Omega)$. Using 
again the same global Lipschitz estimates to $w_n$, we infer that 
%and the fact that the operator $\Ppo(D^2\cdot)+\mu\cdot$ satisfies the minimum principle for any $\mu$, we infer 
%\end{color}
the sequence $(w_n)_{n\in\mathbb N}$ is equi-Lipschitz. Then there is a subsequence converging to a solution $\underline w$ of 
\begin{equation*}
\left\{\begin{array}{cl}
\Ppo(D^2\underline w)+\mu \underline w=f^+ &\mbox{in}\ \Omega\\
\underline w=0 &\mbox{on}\ \partial\Omega.
\end{array}
\right.
\end{equation*} %\bigskip

Now we assume $\mu<\mu_1^+$ and 
consider general $f$. The above functions $\underline w$ and $\overline w$ are respectively sub and supersolution of \eqref{dir1}. 
%Since the operator $\Ppo(\cdot)$ is subadditive, then for $\mu<\mu^+_1$ the comparison principle applies to  $\Ppo(D^2\cdot)+\mu\cdot$. 
To apply the Perron method, we introduce 
\[
\cW=\left\{w\in {\rm C}(\ol\Omega):\;\underline w\leq w\leq\overline w
\;\;\text{ and } w  \text{ supersolution of }\eqref{dir1}\right\}.
\]
and, arguing as in proving the equi-continuity of $(w_n)$ in the case  
$h=-f^-$, 
observe by the local estimates of Lemma \ref{Lip-est} that 
$\cW$ is equi-continuous on $\ol\Omega$. 
%Now we consider the following sequence of solutions, $u_0(x)\equiv 0$
%\begin{equation}\label{dir4}
%\left\{\begin{array}{lc}
%\Ppo(D^2u_{n+1})=f-\mu u_n &\mbox{in}\ \Omega\\
%u_{n+1}=0 &\mbox{on}\ \partial\Omega.
%\end{array}
%\right.
%\end{equation}
%These  solutions exists, we want to prove that, for any $n$,
%$$\underline w\leq u_n\leq \overline w.$$
%This will be done by induction. Suppose that it is true for $u_n$, we want to prove it for $u_{n+1}$.
%
%
%
 %Observe that
%$$
%\left\{\begin{array}{lc}
%\Ppo(D^2u_{n+1})=f-\mu u_n\geq -f^- -\mu \overline w &\mbox{in}\ \Omega\\
%u_{n+1}=0 &\mbox{on}\ \partial\Omega.
%\end{array}
%\right.
%$$
%By comparison this implies that  $u_{n+1}\leq \overline w$. Analogously $u_{n+1}\geq \underline w$. 
%This implies that the sequence $u_n$ is equi-bounded and equi-continuous and then we 
%can pass to the limit.
Setting 
$$
u(x)=\inf\{w(x):\; w\in\cW \},
$$ 
we get a continuous function on $\ol\Omega$, which solves \erf{dir1} due to the Perron method. 
%The uniqueness is an immediate consequence of the minimum principle and maximum principle.
\end{proof}

\begin{theorem}\label{eigenfunction}
 Let $\Omega$ be a strictly convex domain. Then there exists a  function $\psi_1\in{\rm C}(\overline\Omega)$ such that
\begin{equation}\label{1dir}
\left\{\begin{array}{cl}
\Ppo(D^2\psi_1)+\mu_1^+\psi_1=0,\, \ \psi_1>0 &\mbox{in}\ \Omega\\
\psi_1=0 &\mbox{on}\ \partial\Omega.
\end{array}
\right.
\end{equation}
\end{theorem}
%
%\begin{corollary}
%$\overline{\mu}_1^+=\mu_1^+$ {\color{blue} (We already know that $\overline{\mu}_k^+=\mu_k^+$, see %Remark 4.8 in \cite{BGI})}
%\end{corollary}
\begin{proof}
Let $\mu_n\nearrow\mu^+_1$ and use Theorem \ref{exi} to build $u_n\in {\rm C}(\overline\Omega)$ the solution of 
\begin{equation}\label{dir5}
\left\{\begin{array}{cl}
\Ppo(D^2u_{n})+\mu_n u_n=-1 &\mbox{in}\ \Omega\\
u_{n}=0 &\mbox{on}\ \partial\Omega.
\end{array}
\right.
\end{equation}
\textbf{Step 1.}\\
We claim that, up to some subsequence, $\lim_{n\to+\infty}\left\|u_n\right\|_\infty=+\infty$. Assume by contradiction that $\sup_{n\in\N}\left\|u_n\right\|_\infty<+\infty$. Reasoning as in the proof of  Theorem \ref{exi} the sequence $(u_n)_{n\in \N}$ is bounded and equicontinuous. Hence, up to a subsequence, it converges to a nonnegative solution $u$ of
\begin{equation*}
\left\{\begin{array}{cl}
\Ppo(D^2u)+\mu_1^+ u=-1 &\mbox{in}\ \Omega\\
u=0 &\mbox{on}\ \partial\Omega.
\end{array}
\right.
\end{equation*}
The function $u$ is positive in $\Omega$, otherwise if $\min_{x\in\overline\Omega}u=u(x_0)=0$ and $x_0\in\Omega$, then $\varphi(x)=0$ should be a test function touching $u$ from below in $x_0$ and therefore should satisfy $0\leq-1$, a contradiction.\\
Hence, for small positive $\varepsilon$, we have
$$\Ppo(D^2u)+(\mu_1^++\varepsilon) u\leq0\quad\;\text{in $\Omega$}$$
contradicting the maximality of $\mu_1^+$.

\medskip
\textbf{Step 2.}\\
For $n\in\N$ the functions $v_n=\frac{u_n}{\left\|u_n\right\|_\infty}$ satisfy
\begin{equation}\label{dir6}
\left\{\begin{array}{cl}
\Ppo(D^2v_{n})+\mu_n v_n=\frac{-1}{\left\|u_n\right\|_\infty} &\mbox{in}\ \Omega\\
v_{n}=0 &\mbox{on}\ \partial\Omega
\end{array}
\right.
\end{equation}
and are bounded. Again by equicontinuity, extracting a subsequence if necessary, $(v_n)_{n\in\N}$ converges uniformly to a nonnegative function $\psi_1$ such that $\left\|\psi_1\right\|_\infty=1$. Taking the limit as $n\to\infty$ in (\ref{dir6}) we have  
\begin{equation*}
\left\{\begin{array}{cl}
\Ppo(D^2\psi_1)+\mu_1^+\psi_1=0 &\mbox{in}\ \Omega\\
\psi_1=0 &\mbox{on}\ \partial\Omega.
\end{array}
\right.
\end{equation*}
By the strong minimum principle (\cite[Remark 2.6]{BGI}), we conclude $\psi_1>0$ in $\Omega$ as we wanted to show.
\end{proof}

\section{Influence of the first order term}
We shall see that in the presence of a first order term, strict convexity may not be enough to have existence. And even in the uniformly convex case, if the first-order term is \lq\lq too large\rq\rq\ there may not be existence of solution.

\subsection{Nonexistence results for strictly convex domain} 

Let $\Omega$ be a bounded convex domain in $\R^{N}$ and $k<N$. 
Assume that 
\[
\Omega\subset \{z=(x,y)\in\mathbb R^{N-1}\times\R\mid y>0\}, \quad 0=(0,0)\in\pl\Omega,
\]
and, as $(x,y)\in \pl\Omega$ and $x\to 0$, 
\begin{equation}\label{eq1}
y=o(|x|^2). 
\end{equation}

\bthm \label{bb}Under the hypotheses above, there are no positive supersolutions 
$u\in \operatorname{LSC}(\Omega)$ of 
\[%\tag{Eq}
\cP_k^+(D^2u)+b|Du|\leq 0 \ \ \ \text{ in }\Omega,
\]
where $b$ is a positive constant, with the property $\displaystyle \lim_{\Omega\ni z\to0}u(z)=0$.  
\ethm

\begin{remark}  It is worth to point out, as a consequence of Theorem \ref{bb},  that there are no positive eigenfunctions (with Dirichlet boundary) associated to $\cP_k^+(D^2\cdot)+b|D\cdot|$ if $b>0$. This striking feature is further emphasized by the positivity of the so called \lq\lq generalized principal eigenvalue\rq\rq $\mu^+_k$, at least if  $\Omega\subseteq B_R$ and $bR<k$. In fact $\mu^+_k\geq\frac{2(k-bR)}{R^2}$. This inequality can be easily deduced by considering $v(x)=R^2-|x|^2$ in the definition of $\mu^+_k$.
\end{remark}

\bproof 
%By the strong minimum principle we know that if $u$ is a nonnegative solution of (Eq), then $u\equiv 0$ or $u>0$ in $\Omega$.
 By contradiction  we suppose 
that there is a supersolution $u\in {\rm C}(\ol\Omega)$ of 
\[
\cP_k^+(D^2u)+b|D u|\leq 0 \ \ \ \text{ in }\Omega,
\]
with  $b>0$, such that $\displaystyle \lim_{z\to0}u(z)=0$ and 
\begin{equation}\label{eq2}
u>0 \ \ \ \text{ in }\Omega. %,
\end{equation}

We may choose, in view of \eqref{eq1}, a constant $R>0
$ and  a function $g\in {\rm C}^2(\R^{N-1})$ such that 
\[
g(0)=0,\quad D g(0)=0,\quad D^2g(0)=0, 
\]
and
\begin{equation}\label{eq3}
\{(x,y)\in \ol B_R((0,0))\stm \{(0,0)\} \mid \ y\geq g(x)\ \} \ \subset\  \Omega. 
\end{equation}
We may moreover assume that 
\begin{equation}\label{eq4}
k|D^2g(x)|<b \ \ \ \text{ for all  }\ x\in\R^{N-1}, \ \text{ with }|x|<R,
\end{equation}
where $|D^2g(x)|=\max_{i}|\gl_i(D^2g(x))|$. 

By \eqref{eq2} and \eqref{eq3}, we have
\[
\rho:=\min\{u(x,y)\mid (x,y)\in\pl B_R((0,0)),\ y\geq g(x)\}>0.
\]

Set 
\[
\Omega_R=\{(x,y)\in B_R((0,0))\mid y>g(x)\},
\]
and note that
\[\bald
&\ol\Omega_R=\{(x,y)\in\ol B_R((0,0))\mid y\geq g(x)\}\subset \ol\Omega, 
\\&\ol\Omega_R\stm\{(0,0)\}\subset\Omega,
\\&
\pl\Omega_R=\{(x,y)\in\pl B_R((0,0))\mid y\geq g(x)\}\ \cup\ 
\{(x,y)\in\ol B_R((0,0))\mid y=g(x)\}. 
\eald\]

Using $\displaystyle \lim_{z\to0}u(z)=0$, we then may select a point $z_0=(x_0,y_0)\in \Omega_R$ (close to the origin) so that 
\[
u(z_0)<\rho. 
\] 
We may as well choose a function $\theta\in C^2(\R)$ so that 
\[
\theta(0)=0,\quad \theta'(r)>0 \ \ \forall r\in\R, \quad \text{ and }\quad \lim_{r\to+\infty}\theta(r)=\rho.
\]

Let $\varepsilon>0$ and set
\[
\theta_\varepsilon(r)=\theta(r/\varepsilon) \ \ \ \text{ for }\ r\in\R,
\]
and 
\[
\phi_\varepsilon(x,y)=\theta_\varepsilon(y-g(x)) \ \ \ \text{ for }\ (x,y)\in\R^{N}.
\]
Consider the function 
\[
\ol\Omega_R \ni z\ \mapsto \ u(z)-\phi_\varepsilon(z),
\]
and note that, for $z=(x,y)\in \pl\Omega_R$, 
\[
u(z)-\phi_\varepsilon(z)\geq
\bcases
u(z)-\theta_\varepsilon(0)\geq 0-0=0 &\text{ if }\ y=g(x),\\[3pt]
\rho-\phi_\varepsilon(z)>\rho-\rho=0 &\text{ otherwise},
\ecases
\]
and, as $\varepsilon\to 0$,  
\[
u(z_0)-\phi_\varepsilon(z_0)=u(z_0)-\theta\left(\frac{y_0-g(x_0)}{\varepsilon}\right)
\ \to \ u(z_0)-\rho<0. 
\]
We fix $\varepsilon>0$ so that 
\[
u(z_0)-\phi_\varepsilon(z_0)<0,
\]
choose a minimum point $z_\varepsilon=(x_\varepsilon,y_\varepsilon)\in\ol\Omega_R$ 
of $u-\phi_\varepsilon$ and note that 
$u(z_\varepsilon)-\phi_\varepsilon(z_\varepsilon)<0$ and, hence, $z_\varepsilon\in\Omega_R$. 
Thus, by the viscosity property of $u$, we have
\[
\cP_k^+(D^2\phi_\varepsilon(z_\varepsilon))+b|D\phi_\varepsilon(z_\varepsilon)|\leq 0, 
\]
where
\[
D\phi_\varepsilon(x,y)=\theta_{\varepsilon}'(y-g(x))(-Dg(x),\,1), 
\]
and
\[
D^2\phi_\varepsilon(x,y)=\theta_\varepsilon'(y-g(x))\bmat -D^2g(x) & 0\\ 0 & 0\emat
+\theta_\varepsilon''(y-g(x))(-D g(x),\,1)\otim (-D g(x),\, 1). 
\]

Let $\xi\in\R^{N-1}$ and $\eta\in\R$, and compute that
\[%\tag{5}
\lan D^2\phi_\varepsilon(x,y)(\xi,\eta),(\xi,\eta)\ran 
=-\theta_\varepsilon'(y-g(x))\lan D^2g(x) \xi,\,\xi\ran
+\theta_\varepsilon''(y-g(x))(-Dg(x) \cdot \xi+\eta)^2.
\]
If $D g(x)=0$, then 
\[\bald
\cP_k^+(D^2\phi_\varepsilon(x,y))&=\sup
\left\{\sum_{i=1}^k\left(-\theta_\varepsilon'(y-g(x))\lan D^2g(x)\xi_i,\xi_i\ran +\theta_\varepsilon''(y-g(x))\eta_i^2\right)\right.\\&\hspace{2cm}\text{ such that }\,\;(\xi_i,\eta_i)\cdot(\xi_j,\eta_j)=\delta_{ij}\Biggr\}.
\eald\]
Taking $\eta_i=0$ and $\xi_i\in\mathbb R^{N-1}$ such that $\xi_i\cdot\xi_j=\delta_{ij}$ for any $i,j=1,\ldots,k$ we get
\[\bald
\cP_k^+(D^2\phi_\varepsilon(x,y))&\geq\sum_{i=1}^k\left(-\theta_\varepsilon'(y-g(x))\lan D^2g(x)\xi_i,\xi_i\ran\right)
\\&
\geq -k\theta_\varepsilon'(y-g(x))|D^2g(x)|. 
\eald\]

Otherwise if $Dg(x)\neq0$, choosing $(\xi_1,\eta_1)=(D g(x),|D g(x)|^2)/\sqrt{|Dg(x)|^2+|D g(x)|^4}$ and $(\xi_2,0),\ldots,(\xi_k,0)$ in such a way 
$\xi_i\cdot\xi_j=\delta_{ij}$ for all $i,j=1,\ldots,k$, we get 
\[\bald
\cP_k^+(D^2\phi_\varepsilon(x,y))%&=\sup
%\left\{\sum_{i=1}^k\left(-\theta_\varepsilon'(y-g(x))\lan D^2g(x)\xi_i,\xi_i\ran +\theta_\varepsilon''(y-g(x))(-Dg(x)\cdot\xi_i+\eta_i)^2\right)\right.\\&\hspace{2cm}\text{s.t.}\,\;(\xi_i,\eta_i)\cdot(\xi_j,\eta_j)=\delta_{ij}\Biggr\}
%\\&
\geq 
\sum_{i=1}^k\left(-\theta_\varepsilon'(y-g(x))\lan D^2g(x)\xi_i,\xi_i\ran\right)
%\\&
\geq -k\theta_\varepsilon'(y-g(x))|D^2g(x)|. 
\eald\]
Since $|D \phi_\varepsilon(x,y)|=\theta_\varepsilon'(y-g(x))\sqrt{|D g(x)|^2+1}\geq \theta_\varepsilon'(y-g(x))$
and $k|D^2g(x_\varepsilon)|<b$ by \eqref{eq4}, 
we obtain the following contradiction:
\[
0\geq \cP_k^+(D^2\phi_\varepsilon(z_\varepsilon))+b|D \phi_\varepsilon(z_\varepsilon)|
\geq \theta_\varepsilon'(y_\varepsilon-g(x_\varepsilon))(b-k|D^2g(x_\varepsilon)|)>0.
\]
\eproof

\noi
%{\bf Remarks. } 
%{\bf 1. {\color{blue}{This could be omitted} }} In the above proof, we do not use the fact that, as $r\to 0+$,  
%\[
%\inf \{|p|\mid p \in D^+u(z),\  z\in\Omega\cap B_r((0,0))\}\ \to \ +\infty.
%\]

%\noi
\begin{remark}\label{remark}
%In case $b=b(x,y)$ and $b(0,0)=0$, %there is a chance to get a non-existence 
The above nonexistence result can be generalized to the nonconstant coefficient case  $b=b(x,y)$ and $b(0,0)=0$ by assuming 
\[
b(x,y)>k|D^2 g(x)|
\]
in a neighbourhood of $(0,0)$. 

 Even in the case where $b$ is constant, we can replace condition \eqref{eq1}
by the condition
\[
b>k|D^2 g(x)|
\]
in a neighbourhood of $(0,0)$. 
\end{remark}

\subsection{Uniformly convex domain with large Hamiltonian}\label{Uniformly convex with large Hamiltonian}
Look at 
\begin{equation}\label{2eq1}
\begin{cases}
\Ppk(D^2u)+b\left|D u\right|=-1 & \text{in $B_R$}\\
u=0 & \text{on $\partial B_R$}.
\end{cases}
\end{equation}
\begin{proposition}\label{nonexistence}
If $0\leq bR<k$, then the problem \eqref{2eq1} has a unique solution which is radial, while if $bR\geq k$ there are no supersolutions.
\end{proposition}

The case $bR>k$ is included in Remark \ref{remark}. In the radial setting the proof is in fact much easier and it includes the case $bR=k$. For the convenience of the reader we report the proof.

\begin{proof}
First, thanks to Proposition \ref{g-comp} (ii), since the right hand side \erf{2eq1} is negative, the comparison principle always holds and solutions of \erf{2eq1} are unique.   

We consider the case $b>0$ and $bR<k$. For $r\in[0,R]$ let
\begin{equation}\label{g}
 g(r)=\frac{r-R}{b}+\frac{k}{b^2}\log\frac{k-br}{k-bR}.
\end{equation}
By a straightforward computation one has 
\begin{equation*}
\begin{split}
k\frac{g'(r)}{r}&+b|g'(r)|=-1\\
\frac{g'(r)}{r}&\geq g''(r)\\
g'(0)&=g(R)=0.
\end{split}
\end{equation*}
Hence $u(x)=g(|x|)$ is the solution of %(let us recall that for $bR\leq k$ comparison principle holds)
\begin{equation}\label{2eq2}
\begin{cases}
\Ppk(D^2u)+b\left|D u\right|=-1 & \text{in $B_R$}\\
u=0 & \text{on $\partial B_R$}.
\end{cases}
\end{equation}
%Here the uniqueness is a consequence of Proposition \ref{g-comp} 
%since $bR\leq k$. 

Let us assume now that $u$ is a supersolution of \eqref{2eq1} and  $bR\geq k$. 
In particular $u>0$ in $B_R$ and  it is a supersolution too in any ball $B_{\frac{ k-\varepsilon}{b}}\subset B_{R}$ for $\varepsilon\in(0,\,k)$.
Let $\varepsilon\in(0,\,k)$ and set 
$$
g_{\varepsilon}(|x|):=\frac{|x|-\frac{k-\varepsilon}{b}}{b}+\frac{k}{b^2}\log\frac{k-b|x|}{\varepsilon},
$$
which, as we have seen above, 
is the solution of \eqref{2eq2} in $B_{\frac{ k-\varepsilon}{b}}$. Since $$u\geq g_\varepsilon\;\;\;\text{on}\;\;\; \partial B_{\frac{ k-\varepsilon}{b}} \qquad\text{and}\qquad b\frac{ k-\varepsilon}{b}<k$$
the comparison principle yields
$$
u(x)\geq g_\varepsilon(|x|)\qquad \text{ for } x\in B_{\frac{ k-\varepsilon}{b}}.
$$
This leads to a contradiction after letting $\varepsilon\to0$, i.e.
$$
u(x)=\infty\qquad \text{ for all }x\in B_{\frac kR}. 
$$
The function $u(x)=\frac{R^2-|x|^2}{2k}$ is the solution of \eqref{2eq1}, with 
$b=0$.  By a direct computation, one can see that 
the solution $g(|x|)$ of \erf{2eq1} with $b>0$, where $g$ is defined in \eqref{g}), converges 
to $u(x)=\frac{R^2-|x|^2}{2k}$ as $b\to 0$. 
%{\color{blue}COMMENT: the case $b<0$ is treated in Proposition \ref{signminus}. Do we want to separate the positive and negative cases? If yes we have to specify $b\geq0$ in the statement of Proposition \ref{nonexistence}.}
\end{proof}

%\red{HI have exchanged the order of Corollary \ref{cor} and Propostion \ref{signminus} (to be deleted)}
\begin{corollary} \label{cor}
Let $\Omega$ be a domain such that $B_R\subset\Omega$. Then
\begin{itemize}
	\item if $bR\geq k$ 
  there aro no positive supersolution of 
$$\Ppk(D^2u)+b\left|D u\right|\leq-1\qquad\text{in \;$\Omega$};$$
\item if $bR>k$ there are no $(\mu,\psi(x))\in\mathbb R_+\times\operatorname{LSC}(\Omega)$ such that 
$$
\Ppk(D^2\psi)+b|D\psi|+\mu\psi\leq 0,\;\;\psi>0 \quad \text{in $\Omega$},
$$ 
i.e. $\overline\mu^+_k=\mu^+_k=0$, where
\begin{equation*}
\begin{split}
\mu^+_k&=\sup\left\{\mu\in\R\;:\;\exists \varphi\in \operatorname{LSC}(\Omega),\;\varphi>0,\;\Ppk(D^2\varphi)+b|Du|+\mu\varphi\leq0\;\,\text{in $\Omega$} \right\}\\
\bar\mu^+_k&=\sup\left\{\mu\in\R\;:\;\exists \varphi\in \operatorname{LSC}(\overline\Omega),\;\varphi>0\;\text{in $\overline\Omega$},\;\Ppk(D^2\varphi)+b|Du|+\mu\varphi\leq0\;\,\text{in $\Omega$} \right\}.
\end{split}
\end{equation*}
\end{itemize}
\end{corollary}
\begin{proof}
The first part directly follows from Proposition \ref{nonexistence}.\\
Assume now by contradiction that there  exist $\mu>0$, $\psi(x)>0$ in $\Omega$ such that 
$$
\Ppk(D^2\psi)+b|D \psi|+\mu\psi\leq 0 \quad \text{in $\Omega$}.
$$
Let $\rho=\frac{k}{b}<R$. Then $B_{\rho}\Subset\Omega$ and $\displaystyle\min_{\overline B_\rho}\psi>0$. Taking $M$ sufficiently large we can guarantee that $u=M\psi$ satisfies
$$
\Ppk(D^2u)+b|D u|\leq-1\quad\text{in $B_\rho$}
$$
which is not possible since $b\rho=k$.
\end{proof}

Now we consider the equation
\begin{equation}\label{eq99}
\Ppk(D^2u)-b|D u|=-1\quad \text{in $B_R$}
\end{equation}
with any $b>0$. 

\begin{proposition} \label{signminus} There exists a unique solution $u\in {\rm C}(\ol B_R)$ of  the Dirichlet problem for \erf{eq99}, with boundary condition $u=0$ on $\pl B_R$.
The solution $u$ is radial. 
\end{proposition}

\bproof  The uniqueness is a consequence of the comparison principle 
(Proposition \ref{g-comp} (ii)). 

The presence of the sign minus in front of $b$  
leads us to look for radial solutions $u(x)=g(|x|)$ of \eqref{eq99} with $g=g(r)$ 
solution of 
\begin{equation}\label{pr1}
\begin{cases}
g''(r)+(k-1)\frac{g'(r)}{r}+bg'(r)=-1 & \text{ for }r\in(0,R]\\
g'\leq0 &\text{ for } r\in[0,R]\\
g''(r)\geq\frac{g'(r)}{r} & \text{ for } r\in(0,R]\\
g'(0)=g(R)=0. &
\end{cases}
\end{equation}
For solving this,  consider the first order problem 
\begin{equation}\label{pr2}
\begin{cases}
h'(r)+(k-1)\frac{h(r)}{r}+bh(r)=-1 & \text{ for }r\in(0,R]\\
h(0)=0
\end{cases}
\end{equation}
whose solution is $$
h(r)=-\frac{e^{-br}}{r^{k-1}}\int_0^re^{bs}s^{k-1}\,ds.
$$
It is clear that 
\begin{equation}\label{eq100}
h(r)<0\quad\text{ for } r\in(0,R]\quad\text{and}\quad  \lim_{r\to0}h(r)=0.
\end{equation}
Moreover by \eqref{pr2} one has 
\begin{equation}\label{eq101}
h'(r)\geq\frac{h(r)}{r}\quad\Longleftrightarrow\quad a(r):=\left(k+br\right)\int_0^re^{bs}s^{k-1}\,ds-r^k e^{br}\geq0\quad\text{ for }r\in(0,R].
\end{equation}
Since $a'(r)=b\int_0^re^{bs}s^{k-1}\,ds\geq0$ and $a(0)=0$, then the inequality on the left hand side of \eqref{eq101} holds true.
Using now \eqref{pr2}--\eqref{eq101} we deduce that  
$$
g(r)=\int_r^R\frac{e^{-bs}}{s^{k-1}}\left(\int_0^se^{bt}t^{k-1}\,dt\right)\,ds
$$
is a solution of \eqref{pr1}, and $u(x)=g(|x|)$ is in turn a solution of \eqref{eq99} such $u=0$ on $\partial\Omega$. 
%\\ As far as the uniqueness of $u$ is concerned, let $v$ be a solution of \eqref{eq99} such $v=0$ on $\partial\Omega$. Consider for $\varepsilon>0$ the function $u_\varepsilon(x)=(1+\varepsilon)u(x)$ which is a $C^2$-solution of 
%\begin{equation}\label{eq102}
%\begin{cases}
%\Ppk(D^2u_\varepsilon)-b|D u_\varepsilon|=-(1+\varepsilon) & \text{in $B_R$}\\
 %u_\varepsilon=0 & \text{on $\partial B_R$}.
%\end{cases}
%\end{equation}
%I%f $\max_{\overline B_R}(v-u_\varepsilon)>0$, then there was $x_\varepsilon\in\Omega$ such that 
%$$
%(v-u_\varepsilon)(x)\leq (v-u_\varepsilon)(x_\varepsilon) \quad\forall x\in\overline B_R.
%$$
%Using $u_\varepsilon$ as a test function for $v$ at $x_\varepsilon$ and using \eqref{eq102} we get
%$$
%-1\leq\Ppk(D^2u_\varepsilon(x_\varepsilon))-b|D u_\varepsilon(x_\varepsilon)|=-(1+\varepsilon),
%$$
%contradiction. Hence $v\leq u_\varepsilon$ in $\overline B_R$. Similarly $v\geq u_{-\varepsilon}$ in $\overline B_R$, where now $u_{-\varepsilon}(x)=(1-\varepsilon) u(x)$. Sending $\varepsilon\to0$ we obtain $v=u$.
%{\color{blue}Is this the unique solution of \eqref{eq99}, even in the case $bR>k$?}
\end{proof}

\subsection{{Case} $bR=k$ with $\Omega=B_R$}\label{diff. eigenvalues}
For $\mu>0$ consider
\begin{equation}\label{subeq1}
\begin{cases}
\Ppk(D^2u)+\frac kR|D u|+\mu u=0,\;\;u>0 & \text{in $B_R$}\\
u=0 & \text{on $\partial B_R$}.
\end{cases}
\end{equation}
Consider moreover the ODE
\begin{equation}
\begin{cases}
k\left(\frac1r-\frac1R\right)\varphi'(r)+\mu\varphi(r)=0& \text{for $r\in(0,R)$}\\
\varphi(0)=a>0, \;\varphi(R)=0,
\end{cases}
\end{equation}
where $a$ is a constant. 
By computations, the solution $\gf=\gf_{\mu,a}$ is given by 
$$
\varphi_{\mu,a}(r)=a{\left(1-\frac rR\right)}^{\frac{\mu R^2}{k}}\exp\left(\frac{\mu R}{k}r\right)
$$
and 
\begin{equation}\label{subeq2}
\varphi_{\mu,a}'(r)<0 \qquad \text{for any $r\in(0,R)$}.
\end{equation}
If in addition 
$$
\mu\leq\frac{k}{R^2}
$$
then
\begin{equation}\label{subeq3}
\frac{\varphi_{\mu,a}'(r)}{r}\geq\varphi_{\mu,a}''(r) \qquad \text{for all $r\in(0,R)$}.
\end{equation}
Combining \eqref{subeq2}--\eqref{subeq3} we deduce that $u_{\mu,a}(x)=\varphi_{\mu,a}(|x|)$ satisfies for any $
\mu\leq\frac{k}{R^2}
$
\begin{equation*}
\Ppk(D^2u_{\mu,a})+\frac kR\left|D u_{\mu,a}\right|+\mu u_{\mu,a}=0 \qquad\text{for $0<|x|<R$}.
\end{equation*}
Moreover by direct computation
$$
D u_{\mu,a}(0)=0,\;\;D^2u_{\mu,a}(0)=-\frac\mu k u_{\mu,a}(0) I,
$$
hence
$$
\begin{cases}
\Ppk(D^2u_{\mu,a}(x))+\frac kR\left|Du_{\mu,a}(x)\right|+\mu u_{\mu,a}(x)=0, \;\;u_{\mu,a}>0 & \text{in $B_R$}\\
u_{\mu,a}=0 & \text{on $\partial B_R$}.
\end{cases}
$$
In particular 
$$
\mu^+_k\geq\frac{k}{R^2},
$$
while, since the maximum principle is violated, we deduce by \cite[Theorem 4.1]{BGI} that 
$$
\overline \mu^+_k=0.
$$
This shows that the equality $\overline \mu^+_k=\mu^+_k$, which holds when $bR<k$, see \cite[Theorem 4.4]{BGI}, may fails as soon as $bR=k$.

\begin{remark}
\rm
Note that %$\mu^+_k\leq\frac{2k(k+2)}{R^2}$.
$\mu^+_k$ is finite since it is bounded from above by the principal eigenvalue of the operator $\Delta\cdot+\frac kR|D\cdot|$.
\end{remark}

\section{More on the weight of the first order problem.}\label{radial weight}
Let us consider the problem
\begin{equation}\label{3eq1}
\begin{cases}
\Ppk(D^2u)+b(r)\left|Du\right|=-1 & \text{in $B_R$}\\
u=0 & \text{on $\partial B_R$}
\end{cases}
\end{equation}
where  $b\in {\rm C}([0,R])\cap C^1(0,R)$ is  a radial function. We aim to generalize the existence results of subsection \ref{Uniformly convex with large Hamiltonian} to this setting and, at least in a model case, see \eqref{H1},  we shall analyze how the solutions of \eqref{3eq1} are affected by the monotonicity changes of $b(r)$. Having in mind the case $b$ constant, roughly speaking  a transition from $b$ negative to $b$ positive force the solutions $u$ to solve  a second order initial value problem near the origin, then a first order boundary value problem.  

%\red{HI have replaced $\tilde r$ by $r_0$. $\tilde r$ and $\bar r$ look similar. (Delete)} 

Concerning $b(r)$  we  assume that there exists $r_0\in[0,R]$ such that 
\begin{equation}\label{H1}
(r-r_0) (rb(r))'\geq0\quad\text{ for }r\in(0,R).
\end{equation}

Note that if $r_0=0$ or $r_0=R$ then \eqref{H1} reduces respectively to 
the cases $(rb(r))'\geq0$ or $(rb(r))'\leq0$ in $(0,R)$, i.e. the constant sign case of $b(r)$.

\textbf{Definition of ${R_0}$.}  We define $R_0\in (0,\,R]$ as follows. \\ 
If $rb(r)<k$ for any $r<R$ then $R_0:=R$.\\
If there exists $r\in(0,R)$ such that $rb(r)=k$ then $R_0:=\inf\left\{r<R\,:\,rb(r)=k\right\}$.

The above definition of $R_0$ makes sense for any $b\in {\rm C}([0,\,R])$ since 
$rb(r)<k$ holds for $r=0$. 

\begin{remark}
{\rm If $b$ is a positive constant then $R_0=\min\left\{\frac kb,R\right\}$.}
\end{remark}

\begin{proposition}\label{3prop1}
Assume condition \eqref{H1}. If 
\begin{equation}\label{3eq2}
\int^{R_0}\frac{r}{k-rb(r)}\,dr<+\infty
\end{equation} 
then $R_0=R$ and problem \eqref{3eq1} has a unique solution, which is radial. On the other hand, if %Otherwise if 
\begin{equation}\label{3eq3}
\int^{R_0}\frac{r}{k-rb(r)}\,dr=+\infty
\end{equation} 
then no supersolutions of \eqref{3eq1} exist in $B_{R_0}$. 
\end{proposition}

\begin{proof}
First we assume condition \eqref{3eq2}. By contradiction let us assume that $R_0<R$. Since $rb(r)\in C^1$ then there %exist  positive $\varepsilon$ and 
exists positive $M$ such that 
$$
k+M(r-R_0)\leq rb(r) \quad\text{in}\quad [R_0/2,\,R_0]. 
%[R_0-\varepsilon,R_0].
$$
This would imply
$$
\int^{R_0}\frac{r}{k-rb(r)}\,dr\geq\int^{R_0}\frac{r}{M(R_0-r)}\,dr=+\infty,
$$
contradiction. Hence $R_0=R$.

As usual, the uniqueness follows from the comparison principle. 

\smallskip
\textbf{Case $r_0\in(0,R)$.}\\
We start by looking for a radial solution $u(x)=g_1(|x|)$ with $g_1=g_1(r)$ solution of 
\begin{equation}\label{eq199}
\begin{cases}
g_1''(r)+(k-1)\frac{g_1'(r)}{r}-bg_1'(r)=-1 \\
g_1'\leq0 \\
g_1''(r)\geq\frac{g_1'(r)}{r} \\
g_1'(0)=0
\end{cases}
\end{equation}
in a neighbourhood of zero.
This leads us to consider the following first order problem, $h_1=g'_1$,
\begin{equation}\label{eq200}
\begin{cases}
h_1'(r)+(k-1)\frac{h_1(r)}{r}-bh_1(r)=-1 \\
h_1'(r)\geq\frac{h_1(r)}{r}\\
h_1(r)\leq0\\
h_1(0)=0.
\end{cases}
\end{equation}
As in the proof of Proposition \ref{signminus}, the function
\begin{equation}\label{eq204}
h_1(r)=-\frac{e^{B(r)}}{r^{k-1}}\int_0^r e^{-B(s)}s^{k-1}\,ds
\end{equation}
where $B'=b$, satisfies \eqref{eq200} and $h_1'\geq\frac{h_1}{r}$
in an interval $[0,\,c]$ 
provided
\begin{equation}\label{eq201}
a(r):=\left(k-rb(r)\right)\int_0^r e^{-B(s)}s^{k-1}\,ds-e^{-B(r)}r^k\geq0 \ \ \text{ in }
[0,\,c]. 
\end{equation}
%in a neighbourhood of the origin. 
Since $a(0)=0$ and 
$$
a'(r)=-(rb(r))'\int_0^r e^{-B(s)}s^{k-1}\,ds\geq0\quad\text{in $[0,\,r_0]$},
$$
then \eqref{eq201} holds for any $r\in[0,\, r_0]$.\\
Now if $a(r)\geq0$ in $[0,R]$, then $h_1$ is a global solution of \eqref{eq200} and 
\begin{equation}\label{eq202}
g_1(r)=-\int_r^Rh_1(s)\,ds
\end{equation}
is the solution \eqref{eq199} in $[0,R]$ satisfying $g_1(R)=0$.\\
If otherwise there exists $\bar r\in[r_0,\,R)$ such that 
\begin{equation}\label{eq205}
a(\bar r)=0 \quad\text{and}\quad a(r)<0 \quad\text{in}\quad [\bar r,R] 
\end{equation}
then the function 
\begin{equation}\label{eq208}
g_2(r)=\int_r^R\frac{s}{k-sb(s)}\,ds
\end{equation}
is well defined by \eqref{3eq2} and it is a solution of
\begin{equation}\label{3eq4}
\begin{cases}
k\frac{g_2'(r)}{r}-b(r)g_2'(r)=-1 & \text{ for } r\in(\bar r,R)\\
g_2'(r)\leq0& \text{ for } r\in(\bar r,R)\\
g_2(R)=0.
\end{cases}
\end{equation}
Moreover, using \eqref{H1}, one has
\begin{equation}\label{eq206}
 g_2''(r)=\frac{g_2'(r)}{r}-\frac{r}{(k-rb(r))^2}(rb(r))'\leq\frac{g_2'(r)}{r}\qquad 
\text{ for } r\in(\bar r,R).
\end{equation}
Let us define 
\begin{equation}\label{eq203}
g(r)=\begin{cases}
g_1(r) & \text{for $r\in[0,\bar r]$}\\
g_2(r) & \text{for $r\in(\bar r, R]$},
\end{cases}
\end{equation}
where $g_1(r)=-\int_r^{\bar r}h_1(s)\,ds+g_2(\bar r)$ and $h_1$ is defined by \eqref{eq204}. By \eqref{eq205} $g(r)\in C^1([0,R])\cap C^2([0,R]\backslash\left\{\bar r\right\})$. We claim that $u(x)=g(|x|)$ is solution of \eqref{3eq1}. Clearly it is a classical solution for any $x\in B_R$ such that $|x|\neq\bar r$.
Moreover note that if $\left.(rb(r))'\right|_{r=\bar r}=0$ then $u(x)$ is in fact  $C^2(B_R)$. So we may assume that $\left.(rb(r))'\right|_{r=\bar r}>0$, hence by construction the only points $x$ that we need to consider are those for which $|x|=\bar r$. Fix such $x_0\in B_R$ and let $\varphi\in C^2(B_R)$ touching $u$ from above at $x$. First we note that, since  the function $g_1(r)$ and $g_2(r)$ are both twice differentiable in a neighbourhood of $\bar r$, using \eqref{eq206} one has
$$
(g_1-g_2)''(\bar r)=-\frac{1}{k-\bar rb(\bar r)}-g_2''(\bar r)>-\frac{1}{k-\bar rb(\bar r)}-\frac{g'_2(\bar r)}{\bar r}=0
$$
hence $g_1\geq g_2$ around $\bar r$. In this way $\varphi$ touches from above $g_2(|x|)$ at $x_0$ and 
$$
D\varphi(x_0)=D g_2(\bar r),\quad D^2\varphi(x_0)\geq D^2g_2(\bar r). 
$$ 
Then using \eqref{3eq4}
\begin{equation*}
\Ppk( D^2\varphi(x_0))\geq\Ppk(D^2g_2(\bar r))=k \frac{g'_2(\bar r)}{\bar r}=1-b(\bar r)|D \varphi(x_0)|,
\end{equation*}
which shows that $u$ is a viscosity subsolution. The supersolution property of $u$ can be proved in a similar way, using in particular that if $\varphi$ touches $u$ from below at $x_0$, then $\varphi$ is in fact a test function for $g_1(|x|)$. 

%Concerning the uniqueness of $u$, note that by comparison with strict inequalities, see \cite[Section 5C]{CIL}, if $v$ is  solution of  \eqref{3eq1} then for any $\varepsilon>0$ we have 
%$(1-\varepsilon)u(x)\leq v(x)\leq(1+\varepsilon)u(x)$ in $\overline B_R$, so $v=u$. 

\textbf{Cases $r_0=0$ or $r_0=R$.} \\
The solution is given by $u(x)=g_2(|x|)$ if $\bar r=0$ where $g_2$ is defined in \eqref{eq208}, while if $\bar r=R$ then $u(x)=g_1(|x|)$ with $g_1$ defined by \eqref{eq202}.

\medskip
This ends the proof of the first part of the proposition.

Now we %consider the assumption 
assume \eqref{3eq3}. 
By contradiction we assume that there exists a supersolution of \eqref{3eq1}. By the definition of $R_0$ one has $\displaystyle\inf_{r\in[0,R_0-\varepsilon]}k-rb(r)>0$ 
for any $\varepsilon\in(0,\,R_0)$. Consider the function
$$
u_\varepsilon(x):=(1-\varepsilon) g_\varepsilon(|x|):=(1-\varepsilon)
\int_{|x|}^{R_0-\varepsilon}\frac{r}{k-rb(r)}\,dr.
$$
It is a classical strict subsolution of \eqref{3eq1}, since
$$
\Ppk(D^2u_\varepsilon(x))+b(|x|)|D u_\varepsilon(x)|\geq %(1-\varepsilon)
\left(k\frac{g'_\varepsilon(|x|)}{|x|}-b(|x|)g'_\varepsilon(|x|)\right)=-(1-\varepsilon).
$$
 By comparison $u(x)\geq u_{\varepsilon}(x)$ in $B_{R_0-\varepsilon}$ which leads to $u=+\infty$ in $B_{R_0}$ by letting $\varepsilon\to0$.
\end{proof}

\begin{remark}
\rm We briefly discuss the  effects of  reversing the inequality \eqref{H1} in Proposition \ref{3prop1}.  Assume
\begin{equation}\label{H1'}
(r-r_0)(rb(r))'\leq0\quad \forall r\in(0,R).
\end{equation} 
Without loss of generality we may assume $r_0\in(0,R)$. \\
In $[0,\,r_0]$ the function
$$
g_1(r)=-\int_0^r\frac{s}{k-sb(s)}\,ds+c_1
$$
 is a solution of 
\begin{equation*}
\begin{cases}
k\frac{g_1'(r)}{r}-b(r)g_1'(r)=-1 & \text{in $(0,\, r_0]$}\\
g_1'(0)=0
\end{cases}
\end{equation*}
for any choice of the constant $c_1$. Moreover $g'_1(r)\leq0$ and $g_1''(r)\leq\frac{g_1'(r)}{r}$ for any $r\in(0,\,r_0]$.\\
In $[r_0,R]$ where $g_1''(r)\geq\frac{g_1'(r)}{r}$  we look at the second order problem
\begin{equation*}
\begin{cases}
g_2''(r)+\frac{k-1}{r}g_2'(r)-b(r)g_2'(r)=-1 & \text{in $[r_0,R]$}\\
g_2''(r)\geq\frac{g_2'(r)}{r},\;g_2'(r)\leq0 & \text{in $[r_0,R]$}\\
g_2(R)=0.
\end{cases}
\end{equation*}
By computations
$$
g_2(r)=\int_r^R\frac{e^{B(s)}}{s^{k-1}}\left(\int_{r_0}^se^{-B(\tau)}\tau^{k-1}\,d\tau+c_2\right)\,ds
$$
where $B'=b$ and any  $c_2\geq\frac{r_0^k e^{-B(r_0)}}{k-r_0b(r_0)}$.
If we fix 
$$
c_2=\frac{r_0^k e^{-B(r_0)}}{k-r_0b(r_0)}\quad\text{and} \quad c_1=\int_0^{r_0}\frac{s}{k-s b(s)}\,ds+g_2(r_0)
$$ 
then the function 
$$
g(r):=\begin{cases}
g_1(r) & r\in[0,r_0]\\
g_2(r) & r\in(r_0,R]
\end{cases}
$$
is in fact of class $C^2$. Then $u(x)=g(|x|)$ is a classical solution of \eqref{3eq1}. This is the main difference with respect to Proposition \ref{3prop1}, where the solution was not  in general $C^2$ in the set  $\partial B_{\bar r}$, see \eqref{eq201}-\eqref{eq205}  for the definition of $\bar r$. This is due to the fact that here we switch from a first order to a second order problem exactly at $r_0$, the point where the derivative of $rb(r)$ vanishes and so $g''_1=g''_2$, while in Proposition \ref{3prop1} this happens at $\bar r>r_0$   where $g''_1(\bar r)\geq g''_2(\bar r)$.

The  uniqueness of solutions of \eqref{3eq1} with \eqref{H1'} 
is due to the comparison principle, as usual.  
%, as well as 

The nonexistence of supersolutions under the assumption \eqref{3eq3},  can be obtained  as in the proof  Proposition \ref{3prop1}. %, by multiplying the solution $u$ constructed above by $1\pm\varepsilon$, so producing strict classical sub and supersolution, then letting $\varepsilon\to0$. 
\end{remark}

\begin{corollary}

Let $\Omega$ be a domain such that $B_R\subset\Omega$ and assume \eqref{H1}. Then
\begin{itemize}
	\item if $\int^{R_0}\frac{r}{k-rb(r)}\,dr=+\infty$ 
  there aro no positive supersolution of 
$$\Ppk(D^2u)+b(r)\left|D u\right|\leq-1\qquad\text{in \;$\Omega$};$$
\item if $R_0<R$ there are no $(\mu,\psi)\in\mathbb R_+\times\operatorname{LSC}(\Omega)$ such that 
$$
\Ppk(D^2\psi)+b(r)|D \psi|+\mu\psi\leq 0,\;\;\psi>0 \quad \text{in $\Omega$},
$$ 
i.e. $\overline\mu^+_k=\mu^+_k=0$. 
\end{itemize}
\end{corollary}
\begin{proof}
First part is a direct consequence of Proposition \ref{3prop1}. Let us assume now that $R_0<R$ and that $\psi$ is a positive supersolution of 
$$
\Ppk(D^2\psi)+b(r)|D \psi|+\mu\psi= 0 \quad \text{in $\Omega$},
$$
with $\mu>0$. Then  the function $\displaystyle u=\frac{\psi}{{\mu\displaystyle\min_{\overline B_{R_0}}\psi}}$ satisfies 
$$
\Ppk(D^2u)+b(r)|D u|= -1 \quad \text{in $B_{R_0}$},
$$
so, by \eqref{3eq3}, $\int^{R_0}\frac{r}{k-rb(r)}\,dr<+\infty$. Hence \eqref{3eq2} implies $R_0=R$, a contradiction.
\end{proof}

\subsection{Case $R=R_0$  and $\Omega=B_R$} 
In  Subsection \ref{diff. eigenvalues} we  showed that the two notions of generalized principal eigenvalues,  $\bar\mu^+_k$ and $\mu_k^+$, does not coincide in the case $bR=k$. This fact still holds in the nonconstant case $b=b(r)$ under some additional assumptions. \\
First the condition $bR=k$ now reads as $Rb(R)=k$. Then we assume
\begin{equation}\label{H2}
l:=\inf_{r\in(0,R)}\frac{(rb(r))'}{r}>0,
\end{equation}
which obviously holds if $b$ is a positive constant.
Note that \eqref{H2}  implies \eqref{H1} with $r_0=0$.\\
For $\mu>0$ let us consider the problem
\begin{equation}\label{2subeq1}
\begin{cases}
\Ppk(D^2u)+ b(r)|D u|+\mu u=0,\;\;u>0 & \text{in $B_R$}\\
u=0 & \text{on $\partial B_R$}.
\end{cases}
\end{equation}
By straightforward computation the functions
$$
\varphi(r)=\varphi(0)\exp\left\{-\mu\int_0^r\frac{s}{k-sb(s)}\,ds\right\}
$$
is a solution of the ODE
\begin{equation*}
\begin{cases}
\left(\frac kr-b(r)\right)\varphi'(r)+\mu\varphi(r)=0& r\in(0,R)\\
\varphi'(0)=0,\;\varphi(0)>0.
\end{cases}
\end{equation*}
Using \eqref{H2} it is easily seen that 
$$
\frac{\varphi'(r)}{r}\geq\varphi''(r) \quad\forall r\in(0,R)
$$
for any $\mu\leq l$. Hence $u(x)=\varphi(|x|)$ are positive radial solution of the equation in \eqref{2subeq1}. If in addition 
$$
\int_0^R\frac{s}{k-sb(s)}\,ds=+\infty
$$
then $u=0$ on $\partial B_R$, leading to 
$$
\overline \mu^+_k=0<l\leq\mu^+_k.
$$

%\appendix
\section{Convex domains}
%We begin by giving a characterization of the domains in $\cC_j$. Given a bounded convex domain $\Omega$ and $x\in\pl\Omega$, we consider the maximal dimension $d_x(\Omega)$ of  linear subspaces $V$ of the tangent space 
%of $\pl\Omega$ at $x$ such that 
%$(x+V)\cap \pl\Omega$ is a neighborhood of $x$ in the relative topology of $x+V$.  
%That is, $d_x(\Omega)$ is the maximum of $m\in\{0,1,\ldots,N-1\}$ such that 
%there exist an $m$-dimensional linear subspace $V$ in $\R^N$ and $\delta>0$ 
%such that $x+V\cap B_\delta \subset\pl\Omega$. 
%We set $d(\Omega)=\max_{x\in\pl\Omega}d_x(\Omega)$. 
%
%\begin{proposition}\label{prop-Cj} Let $\Omega$ be a bounded convex domain. We have 
%$\Omega\in\cC_j$ if and only if $d(\Omega)\leq N-j$. 
%\end{proposition} 

%We need the following facts about strictly convex domains.
%Let
%$e_1,\ldots,e_N$ be the standard basis of $\mathbb R^N$.
%
%
%\begin{definition} A domain $\Omega\subset\R^N$ is strictly convex if 
%\[ 
%(1-t)x+ty\in\Omega\ \ \ \text{ for all } \ 
%x,y\in\ol\Omega, \text{ with } x\not=y, \ 0<t<1.
%\]
%\end{definition}

In order to prove Theorem \ref{thm-Cj} we need the following lemmas.

\begin{lemma} \label{const-omega}Let $K$ be a compact subset of $\R^m$.
Assume that 
\[
0\in K \ \ \AND \ \ K\stm \{0\} \subset \{x=(x_1,\ldots,x_m)\in\R^m\mid x_m<0\}. 
\]
Then there exists a bounded, open, strictly convex set $\omega\subset\R^m$ such that 
\[
K\subset \ol\omega \ \ \AND \ \ \omega\subset \{x\in\R^m\mid x_m<0\}. 
\] 
\end{lemma}

\begin{lemma}\label{relative int} Let $A$ be a compact convex subset of 
$\R^N$ such that $0\in A$. Let $V$ be the linear span of $A$ and set 
$m=\dim V$. Then there exist a basis $\{a_1,\ldots,a_m\}$ of $V$ such that 
$a_i\in A$ for all $i\in\{1,\ldots,m\}$.  
\end{lemma}

Assuming the lemmas above, we first present the proof of Theorem \ref{thm-Cj}. Henceforth $e_1,\ldots,e_N$ will denote the standard basis of $\mathbb R^N$.

\bproof[Proof of Theorem \ref{thm-Cj}]  Assume that $\Omega\in\cC_j$. Fix any $x\in\pl\Omega$ and prove that $d_x(\Omega)\leq N-j$. There is a
$(O,C)\in S_j(\Omega)$, with $C=\omega\tim\R^{N-j}$, such that $x\in\pl(OC)$ 
and $\Omega\subset OC$. Suppose by contradiction that $d_x(\Omega)>N-j$ and set $m=d_x(\Omega)$. 
There exist an $m$-dimensional linear subspace $V$ in $\R^N$ 
and $\delta>0$ such that $x+V\cap B_\delta \subset \pl\Omega$. 
Observe that  
\[x+V\cap B_\delta\subset \pl\Omega \subset \ol{OC}=O\ol C,
\]
and hence
\begin{equation}\label{OTV}
O^Tx+O^T V\cap B_\delta\subset O^TO\ol C=\ol C=\ol\omega\tim \R^{N-j}. 
\eeq
Since $x\in\pl(OC)=O\pl C=O(\pl\omega \tim\R^{N-j})$, we have $O^T x\in 
\pl\omega\tim\R^{N-j}$. Set 
$y=O^Tx$ and $W=O^T V$ and note that $y\in\pl\omega\tim\R^{N-j}$ and $W$ is $m$-dimensional. 

Since $m>N-j$, the $m$-dimensional ball 
$W\cap B_\delta$  is not contained in $\{0^j\}\tim \R^{N-j}$, where $0^j:=(0,\ldots,0)\in\R^j$. Hence, there exists $w \in W\cap B_\delta\stm \{0^j\}\tim \R^{N-j}$, which 
also means that $-w \in W\cap B_\delta\stm \{0^j\}\tim \R^{N-j}$. 
We set $w^{(j)}=(w_1,\ldots,w_j,0,\ldots,0)\in\R^N$
and note that $w^{(j)}\not=0$ since $w\not\in\{0^j\}\tim\R^{N-j}$. 
Moreover, we observe by \erf{OTV} that 
\[
y\pm w\in y+W\cap B_\delta\subset \ol\omega \tim \R^{N-j},
\]
and hence, $y^{(j)}\pm w^{(j)} \in\ol\omega$. Since $\omega$ is strictly convex and 
$w^{(j)}\not=0$, it is obvious 
that
\[
y=\frac{1}{2}(y+w^j+y-w^j)\in\omega\tim\R^{N-j}.
\]
This contradicts that $y\in\pl\omega\tim\R^{N-j}$. Thus, we have shown that $d(\Omega)\leq N-j$. 
\smallskip

Next, we assume that $d(\Omega)\leq N-j$. Fix any $z\in\pl\Omega$ and $\nu\in N_\Omega(z)$.  
By translation, we may assume that $z=0$. Set 
\[
A=\ol\Omega\cap\{x\in\R^N\mid \nu\cdot x\geq 0\},
\]
and note that $0\in A\subset \pl\Omega$ and $A$ is a compact  convex set. Consider the linear 
span $V_0$ of $A$. It follows that $\dim V_0\leq d(\Omega)$. Indeed, by Lemma \ref{relative int}, there exists a linear basis $\{a_1,\ldots,a_m\} \subset A$ of $V_0$. 
Set $a=(a_1+\cdots+a_m)/m$ and observe that $a\in A\subset\pl\Omega$ 
and, for $\delta>0$ small enough,  
\[
a+B_\delta\cap V_0=B_\delta(a)\cap V_0\subset \Big\{\sum_{i=1}^m t_ia_i\mid  t_i\geq 0,\, \sum_{i=1}^m t_i\leq 1\Big\} 
\subset A\subset \pl\Omega,
\]
which ensures that $\dim V_0\leq d(\Omega)$. 

Since $A$ is included in the supporting plane $\{x\in\R^N \mid\nu\cdot x=0\}$ of $\Omega$ at $0$, which is $(N-1)$-dimensional, we may choose a $(N-j)$-dimensional 
subspace $V$ of $\{x\in\R^N \mid\nu\cdot x=0\}$ such that $A\subset V_0\subset V$. 

Now, we observe that
\begin{equation}\label{strict convex V perp} 
\nu\cdot x<0 \ \ \text{ for all }x\in \ol\Omega\stm V. 
\eeq
Indeed, if $x\in\ol\Omega\stm V$, then $x\in\ol\Omega\stm A$ and, by the definition of  
$A$, $\nu\cdot x<0$. 

By orthogonal transformation, we may assume that 
\[
\nu=e_j \ \ \AND \ \ V=\{0^j\}\tim \R^{N-j}.  
\]
Set 
\[
K=\{(x_1,\ldots,x_j)\in\R^j\mid (x_1,\ldots,x_N)\in \ol\Omega \text{ for some }(x_{j+1},\ldots,x_N)\in\R^{N-j}\}. 
\]
This $K$ is the projection of $\ol\Omega$ onto $\R^j$ and is compact and convex. 
Clearly, we have $0\in K$.  Moreover,  
if $x\in K\stm \{0^j\}$, 
then, by the definition of $K$, there is $y=(y_1,\ldots,y_N)\in\ol\Omega$ such that 
$x=(y_1,\ldots,y_j)$ and we have $y\not \in V=\{0^j\}\tim \R^{N-j}$ 
since $x\not=0^j$, 
and, by \erf{strict convex V perp},  $\nu\cdot y<0$, which reads $y_j<0$. 
We may apply Lemma \ref{const-omega}, to conclude that 
there is a bounded strictly convex domain $\omega\subset\R^j$ such that 
\begin{equation}\label{last condition}
K\subset \ol\omega \ \ \AND \ \ 
\omega \subset \{x=(x_1,\ldots,x_j)\in \R^j\mid x_j<0\}. 
\eeq
Hence, by the definition of $K$, we see that 
\[
\ol\Omega \subset \ol\omega \tim\R^{N-j},
\] 
which implies, since $\Omega$ and $\omega$ are nonempty convex sets
that
\[
\Omega\subset \omega\tim\R^{N-j}. 
\]
It is obvious from \erf{last condition} that for the boundary point $0$ of $\Omega$, 
$0\in \pl\omega\tim\R^{N-j} =\pl(\omega\tim\R^{N-j})$. 
\end{proof}

We need the following lemma for the proof of Lemma \ref{const-omega}.

\begin{lemma} \label{ball u-convexity} Let $B_R(z)$ be the open ball of $\R^N$ with radius $R>0$ and center $z$. For any $x,y\in\ol B_R(z)$, with $x\neq y$, and $0<t<1$, 
there exists a positive constant $\delta=\delta(R,t,|x-y|)$ such that 
\[
B_\delta(tx+(1-t)y)\subset B_R(z). 
\]
Moreover, $\delta(R,t,|x-y|)$ can be chosen depending only on $R$, $t$, and $|x-y|$ 
and decreasingly on $R$.  
\end{lemma} 

\bproof Combine
\[
|tx+(1-t)y-z|^2=t^2|x-z|^2+(1-t)^2|y-z|^2+2t(1-t)(x-z)\cdot (y-z),
\]
and
\[
|x-y|^2=|x-z-(y-z)|^2=|x-z|^2+|y-z|^2-2(x-z)\cdot (y-z),
\]
to get 
\[\bald
|tx+(1-t)y-z|^2&\,= t^2|x-z|^2+(1-t)^2|y-z|^2+t(1-t)(|x-z|^2+|y-z|^2-|x-y|^2)
\\&\,=t|x-z|^2+(1-t)|y-z|^2-t(1-t)|x-y|^2\leq R^2-t(1-t)|x-y|^2. 
\eald\]
Hence, if we set 
\[
\delta:=\frac 12 \left(R-\sqrt{R^2-t(1-t)|x-y|^2}\right)
=\frac{t(1-t)|x-y|^2}{2(R+\sqrt{R^2-t(1-t)|x-y|^2})},
\]
then we have $B_\delta(tx+(1-t)y)\subset B_R(z)$. The choice $\delta$ above 
has the required dependence on $R$, $z$ and so on. 
\end{proof}

\bproof[Proof of Lemma \ref{const-omega}] Fix an $R_0>0$ so that $K\subset B_{R_0}$ and for $R\geq R_0$, set
\[
\rho(R) =\sup\{h\geq 0 \mid K\subset \ol B_R(-h e_m)\},
\]
where $e_m$ is the unit vector in $\R^m$, with unity as the last ($m$-th) entry. 
Since $0\in K$, we see that $\rho(R)\leq R$. Also, since 
$K\subset B_{R_0}\subset B_R(-(R-R_0)e_m)$, we have $\rho(R)\geq R-R_0$. 
It is now clear that $\rho(R)$ is achieved. 
%is the {\color{blue}maximum} of those $h\geq 0$ which satisfy $K\subset \ol B_R(-he_m)$.  
In particular, if $S>R\geq R_0$, then 
$K\subset \ol B_R(-\rho(R)e_m)\subset \ol B_S(-\rho(R)e_m)$ and hence, $\rho(S)\geq \rho(R)$. Thus, $\rho(R)$ depends on $R$ nondecreasingly (in fact, increasingly).    

We claim that 
\begin{equation}\label{Rinfty}
\lim_{R\to \infty}(R^2-\rho(R)^2)=0. 
\eeq

To prove this, fix first  any $r>0$. Since $K\stm B_r$ is compact and 
$K\stm B_r\subset \{x\in\R^m\mid x_m<0\}$, we may choose $0<\gamma_1<R_1$ so that 
\[
K\stm B_r\subset \{x\in\R^m\mid x_1^2+\ldots+x_{m-1}^2<R_1^2,\, -R_1<x_m<-\gamma_1\}.
\] 
One can always replace $R_1$ and $\gamma_1$, without 
violating the above inclusion, by larger and smaller ones, respectively. 
In what follows, we may fix $R_1$ so that $R_1>r$ and consider $0<\gamma<\gamma_1$.

We next choose $0<h<R$ such that %the intersections of the ball $\ol B_R(-he_m)$ 
%the hyperplane $x_m=0$ and $x_m=-\gamma$ make 
%$(m-2)$-dimensional spheres, respectively, 
%\[
%\{x\in\R^m\mid x_1^2+\cdots+x_{m-1}^2=r^2,\ x_m=0\}, 
%\]
%and
%\[
%\{x\in\R^m\mid x_1^2+\cdots+x_{m-1}^2=R_1^2,\ x_m=-\gamma\}.
%\]
\begin{equation*}
\ol B_R(-he_m)\cap\left\{x\in\R^m\mid x_m=0\right\}=\left\{x\in\R^m\mid x_1^2+\cdots+x_{m-1}^2\leq r^2,\ x_m=0\right\}
\end{equation*}
and
\begin{equation*}
\ol B_R(-he_m)\cap\left\{x\in\R^m\mid x_m=-\gamma\right\}=\left\{x\in\R^m\mid x_1^2+\cdots+x_{m-1}^2\leq R_1^2,\ x_m=-\gamma\right\}
\end{equation*}
i.e. we choose 
%\begin{equation}\label{Rinfty1}
%R^2=h^2+r^2=(h-\gamma)^2+R_1^2.
%\eeq
%Solving this, we have
\[
h=\frac{R_1^2+\gamma^2-r^2}{2\gamma} \ \ \text{ and } \ \ R=\sqrt{h^2+r^2}=\sqrt{(h-\gamma)^2+R_1^2}\,.
\]

Reducing  $\gamma_1$ if necessary, we can suppose that the function $g(\gamma)=\frac{R_1^2+\gamma^2-r^2}{2\gamma}$ is decreasing  in $(0,\gamma_1]$. Let $h_1=g(\gamma_1)$.
%\[
%\lim_{\gamma\to 0}\frac{R_1^2+\gamma^2-r^2}{2\gamma}=+\infty,
%\]
%we infer by continuity that there exists $h_1>0$ such that if $h>h_1$, then
%there exists $\gamma\in(0,\,\gamma_1)$ such that 
%\begin{equation}\label{Rinfty2}
%\frac{R_1^2+\gamma^2-r^2}{2\gamma}=h. 
%\eeq
Fix any $R>\sqrt{h_1^2+r^2}$ %Choose $h>h_1$ and $\gamma>\gamma_1$ 
%so that $R=\sqrt{h^2+r^2}$ and \eqref{Rinfty2} holds. 
and let $h=\sqrt{R^2-r^2}>h_1$. Since $g(\gamma)$ is continuous  and $\displaystyle\lim_{\gamma\to0^+}g(\gamma)=+\infty$, there exists $\gamma\in(0,\gamma_1]$ such that
\begin{equation}\label{Rinfty2}
\frac{R_1^2+\gamma^2-r^2}{2\gamma}=h. 
\eeq
%In particular, we have $R\geq R_1$. 
Simple geometry  tells us that 
\[\bald
K&\,\subset B_r\cap\{x\in\R^m\mid x_m\leq0\} \ \cup \ 
\{x\in\R^m\mid x_1^2+\cdots+x_{m-1}^2<R_1^2,\, -R_1<x_m<-\gamma\}
\\&\,\subset B_R(-he_m).
\eald
\]
This inclusion ensures that $\rho(R)\geq h$ and hence 
$R^2-\rho(R)^2\leq R^2-h^2=r^2$.  Hence, we have 
$R^2-\rho(R)^2\leq r^2$ for any $R>\sqrt{h_1^2+r^2}$. Since $r>0$ is arbitrary, 
we conclude \eqref{Rinfty}. 

In case when $\rho(R)=R$ for some $R\geq R_0$, we fix such an $R\geq R_0$ and 
set 
\[
\omega:=B_R(-\rho(R)e_m)=B_R(-Re_m).
\]
It is clear that $\omega$ is a bounded, open, strictly convex subset of $\R^m$ 
and that $K\subset\ol\omega$ and $\omega \subset\{x\in\R^m\mid x_m<0\}$. 

Now, we consider the general case. 
We set 
\[
\Delta=\bigcap_{R\geq R_0}\ol B_R(-\rho(R)e_m).
\]
It is obvious that $\Delta$ is compact and convex and that $K\subset\Delta$.  
Since $K\subset B_{R_0}$, $K\stm\{0\}\subset \{x\in\R^m\mid x_m<0\}$ and $K$ is compact, it is easily seen that $\rho(R_0)>0$. 
Moreover, since $0\in \ol B_R(-\rho(R)e_m)$ for all $R\geq R_0$ and $R\mapsto \rho(R)$ is nondecreasing, we find that $
-\rho(R_0)e_m\in \Delta$. 

We define $\omega$ as the interior $\INT\Delta$ of $\Delta$. 
We need only to show that $\omega \neq \emptyset$, 
which implies by the convexity of $\Delta$ that $\ol\omega =\Delta$, 
and also that $\omega$ is strictly convex  
and contained in $\{x\in\R^m\mid x_m<0\}$. 

For this, we first check that $\omega\subset\{x\in\R^m\mid x_m<0\}$. 
It is enough to show that 
\begin{equation}\label{suplevel}
\Delta\cap \{x\in\R^m\mid x_m\geq 0\}=\{0\}. 
\eeq
Fix any $x\in \Delta$, with $x_m\geq 0$ and note that for $R\geq R_0$,
\[
R^2\geq |x+\rho(R)e_m|^2=|x|^2+2\rho(R)x_m+\rho(R)^2\geq |x|^2+\rho(R)^2,
\]
and, accordingly, 
\[
R^2-\rho(R)^2\geq |x|^2. 
\]
Hence, \erf{Rinfty} implies that $x=0$.  

Next, fix any $\varepsilon>0$ and set
\begin{equation}\label{Rep}
R_\varepsilon=R_0+\frac{R_0^2-\rho(R_0)^2+2\varepsilon\rho(R_0)}{2\varepsilon}
\eeq
and
\[
\Delta_\varepsilon=\bigcap_{R_0\leq R\leq R_\varepsilon}\ol B_R(-\rho(R)e_m). 
\]
We observe that for any $x\in \Delta_\varepsilon\cap\{y\in\R^m\mid y_m<-\varepsilon\}$, if $R>R_\varepsilon$, then we have
\[\bald
|x+\rho(R)e_m|^2&\,=|x+\rho(R_0)e_m+(\rho(R)-\rho(R_0))e_m|^2
\\&\,
\leq R_0^2+(\rho(R)-\rho(R_0))^2 +2(\rho(R)-\rho(R_0))(x+\rho(R_0)e_m)\cdot e_m
%\\&\,=R_0^2+(\rho(R)-\rho(R_0))^2 +2(\rho(R)-\rho(R_0))(x_m+\rho(R_0)) 
%\ (\red{\text{to be deleted}})
\\&\,\leq R_0^2+(\rho(R)-\rho(R_0))^2 +2(\rho(R)-\rho(R_0))(-\varepsilon+\rho(R_0))
%\\&\,=R_0^2+(\rho(R)-\rho(R_0)+\rho(R_0)-\varepsilon)^2-(\rho(R_0)-\varepsilon)^2 \ (\red{\text{to be deleted}})
%\\&\,=R_0^2+(\rho(R)-\varepsilon)^2-(\rho(R_0)-\varepsilon)^2\ (\red{\text{to be deleted}})
\\&\,= R_0^2+\rho(R)^2-\rho(R_0)^2-2\varepsilon(\rho(R)-\rho(R_0)),
\eald\]
and, since 
\[
\rho(R)\geq R-R_0>R_\varepsilon-R_0=\frac{R_0^2-\rho(R_0)^2+2\varepsilon\rho(R_0)}{2\varepsilon},
\]
\[
|x+\rho(R)e_m|^2\leq R_0^2+\rho(R)^2-\rho(R_0)^2-(R_0^2-\rho(R_0)^2)=
\rho(R)^2\leq R^2.
\]
Hence, we find that
\[
\Delta_\varepsilon \cap\{x\in\R^m\mid x_m<-\varepsilon\}\subset \Delta\cap
 \{x\in\R^m\mid x_m<-\varepsilon\}.
\]
The reverse inclusion is trivial and thus we have
\begin{equation} \label{level<-ep}
\Delta_\varepsilon \cap\{x\in\R^m\mid x_m<-\varepsilon\}= \Delta\cap 
\{x\in\R^m\mid x_m<-\varepsilon\}.
\eeq

Let $x,y\in\Delta \stm\{0\}$, with $x\not= y$ and $0<t<1$. 
By \erf{suplevel}, we may select $\varepsilon>0$ so that $x_m,y_m<-2\varepsilon$. 
It follows that $tx_m+(1-t)y_m<-2\varepsilon$. Define $R_\varepsilon>R_0$ by \erf{Rep}. 
Since 
\[
x,y\in \bigcap_{R_0\leq R\leq R_\varepsilon}B_R(-\rho(R)e_m), 
\]
thanks to Lemma \ref{ball u-convexity},  
we can choose $\delta\in (0,\,\varepsilon/2)$ such that 
\[
B_\delta(tx+(1-t)y)\subset\bigcap_{R_0\leq R\leq R_\varepsilon}\ol B_R(-\rho(R)e_m),
\]
which readily yields
\[
B_\delta(tx+(1-t)y)\subset\bigcap_{R_0\leq R\leq R_\varepsilon}\ol B_R(-\rho(R)e_m)
\cap\{z\in\R^m\mid z_m<-\varepsilon\}. 
\]
Using identity \erf{level<-ep}, we find that 
\begin{equation}\label{strictconv<0}
B_\delta(tx+(1-t)y)\subset \Delta. 
\eeq
In particular, this, with $x=-\rho(R_0)e_m$ and $y=-(\rho(R_0)/2)e_m$, ensures that 
$\omega\not=\emptyset$. 

Inclusion \erf{strictconv<0} implies the strict convexity of $\omega$. 
Indeed, let $x,y\in\Delta$, with $x\neq y$, and $0<t<1$. If $x, y$ are both not zero, 
then \erf{strictconv<0} shows that $tx+(1-t)y\in\omega$. Otherwise, we may assume that $y=0$.   
Note  
that $z:=(t/2)x\in\Delta$, $z\neq 0$ and 
\[
tx=\frac{t}{2-t}x+\left(1-\frac{t}{2-t}\right)z
\]
and apply \erf{strictconv<0}, to conclude that $tx\in\omega$.
Thus, we find that $\omega$ is strict convex and completes the proof. 
\end{proof}

\bproof[Proof of Lemma \ref{relative int}] If $A=\{0\}$, then the conclusion of the lemma is obvious since $V=\{0\}$ and $\dim V=0$. (As usual, we agree that the linear span of $\emptyset$ is $\{0\}$.) Assume that $A\not =\{0\}$. Consider all the 
collections $\{b_1,\ldots,b_j\}$ of linearly independent vectors $b_i\in A$. Obviously, 
we have $1\leq j \leq N$ for any such collection $\{b_1,\ldots,b_j\}$. Select 
a such collection
$\{a_1,\ldots,a_k\}$, with maximum number of elements $k$. 
Since $\{a_1,\ldots,a_k\}\subset A$, it follows that $\{a_1,\ldots,a_k\}\subset V$, 
and $\Span\{a_1,\ldots,a_k\}\subset V$. 
Suppose for the moment that $\Span\{a_1,\ldots,a_k\}\not= V$, which implies that  
$A\stm \Span\{a_1,\ldots,a_k\}\not=\emptyset$. 
Then there exists $a_{k+1}\in A\stm \Span\{a_1,\ldots,a_k\}$, which means that 
$\{a_1,\ldots,a_{k+1}\}\subset A$ is a collection of linearly independent vectors. 
This contradicts the choice of $\{a_1,\ldots,a_k\}$ and proves that 
$\Span\{a_1,\ldots,a_k\}= V$ (as well as $k=m$).  
\end{proof}

In a recent article \cite{BR}, Blanc and Rossi have  introduced the following notion. 
%\red{
Here, unlike \cite{BR}, we are only concerned with convex domains. %}

\begin{definition}{\textbf{(${\mathcal G}_j$ condition)}}\\
Let $j\in\left\{1,\ldots,N\right\}$ and let $\Omega\subset\RN$ be a bounded convex domain. We say that 
 $\Omega\in{\mathcal G}_j$   if for any $y\in\partial\Omega$ and any $r>0$ there exists $\delta>0$ such that for every $x\in B_{\delta}(y)\cap\Omega$ and $S\subset\RN$ subspace with $\dim S=j$, then there exists a unit vector $v\in S$ such that 
\begin{equation}
\left\{x+tv\right\}_{t\in\R}\cap B_r(y)\cap\partial\Omega\neq\emptyset.
\end{equation}
\end{definition}
In \cite{BR} they consider the problem
$$\left\{\begin{array}{cc}
\lambda_j(D^2u)=0 &\mbox{in}\ \Omega\\
u=g &\mbox{on}\ \partial\Omega
\end{array}\right.
$$
and they prove that 
if $\Omega\in{\mathcal G}_j\cap {\mathcal G}_{N-j}$ then the above Dirichlet problem is solvable for any $g$
while, if $\Omega$ is not in ${\mathcal G}_j\cap {\mathcal G}_{N-j}$ then there may be some $g$ for which the 
problem is not solvable. This problem is very much related with the results in the present article hence we prove the 
following equivalence.

\begin{proposition} When $\Omega$ is %\red{
bounded, open and %} 
convex, 
$\Omega\in{\mathcal G}_j$ if and only if $\Omega\in{\mathcal C}_{N-j+1}$.
\end{proposition}
%\begin{color}{red}
\begin{proof} The property that $\Omega\not\in\cG_j$ can be stated as follows: 
there exist $y\in\pl\Omega$ and $r>0$ such that for any $\delta>0$, there 
exist $x\in B_\delta(y)\cap \Omega$ and a linear subspace $S$ of $\R^N$, 
with $\dim V=j$,  
for which 
\[
\{x+tv\}_{t\in \R}\cap B_r(y)\cap \pl\Omega=\emptyset \ \ \text{ for any }v\in S, 
\text{ with }|v|=1.
\]
This equality reads
\[
(x+S)\cap B_r(y)\cap \pl\Omega=\emptyset,
\]
and moreover, since $x\in\Omega$, 
\begin{equation} \label{negativeGj}
(x+S)\cap B_r(y)\subset\Omega. 
\eeq

Now, we assume that $\Omega\not\in\cG_j$ and prove that $d(\Omega)\geq j$. 
By the above consideration, there exist $y\in\pl\Omega$ and $r>0$ such that 
for each $k\in\N$, there exist $x\in B_{1/k}(y)\cap\Omega$ and a linear 
subspace $S_k\subset\R^N$, with $\dim S_k=j$, such that \erf{negativeGj} 
holds with $x=x_k$ and $S=S_k$. 

Noting that $\lim_{k\to\infty}x_k=y$ and taking limit as $k\to \infty$ along 
an appropriate subsequence, we can find a linear subspace $S\subset\R^N$, 
with $\dim S=j$, such that 
\begin{equation}\label{negativeGjlimit}
(y+S)\cap B_r(y)\subset\ol\Omega. 
\eeq
(Here, regarding the convergence of $S_k$, one may fix an orthonormal basis 
$\{v_{k,1},\ldots,v_{k,j}\}$ of $S_k$ for each $k$ and look for 
a subsequence of the $k$ for which $\{v_{k,1},\ldots,v_{k,j}\}$ converge in $\R^{N\tim j}$.)
Since $(y+S)\cap B_r(y)$ is a $j$-dimensional ball, with center $y\in\pl\Omega$ 
and $\Omega$ is convex, it is easily seen by \erf{negativeGjlimit} that 
\[
(y+S)\cap B_r(y)\subset \pl\Omega,
\]
which shows that $d(\Omega)\geq j$. 

Next, we assume that $d(\Omega)\geq j$ and prove that $\Omega\not\in \cG_j$. 
This assumption implies that there exist $y\in\pl\Omega$ and a linear subspace $S\subset \R^N$, with $\dim S\geq j$, such that 
\[
y+S\cap B_r\subset\pl\Omega. 
\]
We may assume by replacing $S$, by a subspace of $S$ if necessary, that $\dim S=j$.  
Since $\pl\Omega\neq\emptyset$, there exists a point $z\in\Omega$. By the convexity 
of $\Omega$, with nonempty interior, we see that 
\[
tz+(1-t)y +S\cap B_{(1-t)r}=t z+(1-t)(y+S\cap B_r)\subset \Omega \ \ \text{ for } t\in(0,\,1).
\]
Hence, for $t\in(0,1/2)$, if we set $x_t=tz+(1-t)y$, then 
\[
(x_t +S)\cap B_{r/2}(x_t)=x_t+S\cap B_{r/2} \subset x_t+S\cap B_{(1-t)r}
\subset \Omega,
\]
and, also, $\lim_{t\to 0}x_t =y$.  This shows that $\Omega\not\in\cG_j$. 
Thus, we see that $\Omega\in\cG_j$ if and only if $d(\Omega)\leq j-1$. 
This observation and Theorem \ref{thm-Cj} assure that 
$\Omega\in\cG_j$ if and only if $\Omega\in\cC_{N-j+1}$. 
\end{proof}
%\end{color}
%%%%%%%%%%%%%%%%%%%%%%%

\end{document}